\documentclass[a4paper,twoside,12pt]{amsart}
\RequirePackage{amsmath,amssymb,xypic}
\theoremstyle{plain}
\newtheorem{theorem}{Theorem}[section]
\newtheorem{corollary}[theorem]{Corollary}
\newtheorem{proposition}[theorem]{Proposition}
\newtheorem{lemma}[theorem]{Lemma}

\theoremstyle{definition}
\newtheorem{definition}[theorem]{Definition}
\newtheorem{definition-proposition}[theorem]{Definition-Proposition}
\newtheorem{example}[theorem]{Example}
\newtheorem{remark}[theorem]{Remark}

\newtheorem{notation}[theorem]{Notation}

\numberwithin{equation}{section}

\renewcommand{\emptyset}{\varnothing}

\newcommand{\Union}{\bigcup\limits}

\newcommand{\C}{\mathbb{C}}
\newcommand{\N}{\mathbb{N}}

\newcommand{\Z}{\mathbb{Z}}

\newcommand{\op}{\mathrm{op}}

\DeclareMathOperator{\id}{id}

\newcommand{\DSum}{\bigoplus}

\newcommand{\plim}[1][]{\mathop{\varprojlim}\limits_{#1}}

\renewcommand{\to}[1][]{\xrightarrow[]{#1}}

\newcommand{\isoto}[1][]{\xrightarrow[#1]{\sim}}
\newcommand{\isofrom}[1][]{\xleftarrow[#1]{\sim}}

\newcommand{\Endo}[1][]{\mathrm{End}_{\raise1.5ex\hbox to.1em{}#1}}

\newcommand{\Hom}[1][]{\mathrm{Hom}_{\raise1.5ex\hbox to.1em{}#1}}

\newcommand{\RHom}[1][]{\mathrm{RHom}_{\raise1.5ex\hbox to.1em{}#1}}

\newcommand{\Ext}[2][]{\mathrm{Ext}_{\raise1.5ex\hbox to.1em{}#1}^{#2}}

\newcommand{\THom}[1][]{\mathrm{THom}_{\raise1.5ex\hbox to.1em{}#1}}

\newcommand{\Mod}{\mathrm{Mod}}

\newcommand{\Tens}[1][]{\mathbin{\otimes_{\raise1.5ex\hbox to-.1em{}#1}}}

\newcommand{\LTens}[1][]{\mathbin{\otimes_{\raise1.5ex\hbox to-.1em{}#1}^{L}}}

\newcommand{\Tor}[2][]{\mathrm{Tor}^{\raise1.5ex\hbox to.1em{}#1}_{#2}}

\def\sha{\mathcal{A}}
\def\shb{\mathcal{B}}

\def\shd{\mathcal{D}}

\def\shh{\mathcal{H}}

\def\shl{\mathcal{L}}
\def\shm{\mathcal{M}}
\def\shn{\mathcal{N}}

\def\shp{\mathcal{P}}

\def\shs{\mathcal{S}}

\def\shu{\mathcal{U}}
\def\shv{\mathcal{V}}

\newcommand{\sect}{\varGamma}

\newcommand{\shendo}[1][]{{\mathcal{E}nd}_{\raise1.5ex\hbox to.1em{}#1}}

\newcommand{\shaut}[1][]{{\mathcal{A}ut}_{\raise1.5ex\hbox to.1em{}#1}}

\renewcommand{\hom}[1][]{{\mathcal{H}om}_{\raise1.5ex\hbox to.1em{}#1}}

\newcommand{\rhom}[1][]{{R\mathcal{H}om}_{\raise1.5ex\hbox to.1em{}#1}}

\newcommand{\ext}[2][]{{\mathcal{E}xt}_{\raise1.5ex\hbox to.1em{}#1}^{#2}}

\newcommand{\thom}[1][]{{T\mathcal{H}om}_{\raise1.5ex\hbox to.1em{}#1}}

\newcommand{\tens}[1][]{\mathbin{\otimes_{\raise1.5ex\hbox to-.1em{}#1}}}

\newcommand{\ltens}[1][]{\mathbin{\otimes_{\raise1.5ex\hbox to-.1em{}#1}^{L}}}

\newcommand{\tor}[2][]{{\mathcal{T}or}^{\raise1.5ex\hbox to.1em{}#1}_{#2}}

\newcommand{\oim}[1]{{#1}_*}

\newcommand{\opb}[1]{#1^{-1}}

\newcommand{\GHom}[1][]{\mathrm{GHom}_{\raise1.5ex\hbox to.1em{}#1}}

\newcommand{\GExt}[2][]{\mathrm{GExt}_{\raise1.5ex\hbox to.1em{}#1}^{#2}}

\newcommand{\FHom}[1][]{\mathrm{FHom}_{\raise1.5ex\hbox to.1em{}#1}}

\newcommand{\ghom}[1][]{{\mathcal{GH}om}_{\raise1.5ex\hbox to.1em{}#1}}

\newcommand{\gext}[2][]{{\mathcal{GE}xt}_{\raise1.5ex\hbox to.1em{}#1}^{#2}}

\newcommand{\fhom}[1][]{{\mathcal{FH}om}_{\raise1.5ex\hbox to.1em{}#1}}

\newcommand{\Gr}{\mathop{\mathrm{Gr}}\nolimits}

\newcommand{\tenstop}[1][]{\mathbin{\hat{\otimes}_{\raise1.5ex\hbox to-.1em{}#1}}}

\newcommand{\homtop}[1][]{\mathcal{L}_{\raise1.5ex\hbox to.1em{}#1}}

\newcommand{\Homtop}[1][]{\mathrm{L}_{\raise1.5ex\hbox to.1em{}#1}}

\newcommand{\D}{\mathcal{D}}

\newcommand{\E}{\mathcal{E}}

\renewcommand{\O}{\mathcal{O}}

\def\absdoim#1{\underline{#1}_*}
\def\reldoim[#1]#2{\underline{#2}_{|{#1}*}}
\def\doim{\@ifnextchar [{\reldoim}{\absdoim}}

\def\absdeim#1{\underline{#1}_*}
\def\reldeim[#1]#2{\underline{#2}_{|{#1}*}}
\def\deim{\@ifnextchar [{\reldeim}{\absdeim}}

\def\absdopb#1{\underline{#1}^{-1}}
\def\reldopb[#1]#2{\underline{#2}_{|{#1}}^{-1}}
\def\dopb{\@ifnextchar [{\reldopb}{\absdopb}}

\def\absboim#1{\underline{\underline{#1}}_*}
\def\relboim[#1]#2{\underline{\underline{#2}}_{|{#1}*}}
\def\boim{\@ifnextchar [{\relboim}{\absboim}}

\def\absbeim#1{\underline{\underline{#1}}_*}
\def\relbeim[#1]#2{\underline{\underline{#2}}_{|{#1}*}}
\def\beim{\@ifnextchar [{\relbeim}{\absbeim}}

\def\absbopb#1{\underline{\underline{#1}}^*}
\def\relbopb[#1]#2{\underline{\underline{#2}}_{|{#1}}^*}
\def\bopb{\@ifnextchar [{\relbopb}{\absbopb}}

\newcommand{\ad}{\operatorname{ad}}

\newcommand{\stack}[1]{\mathsf{#1}}

\newcommand{\stka}{\stack{A}}

\newcommand{\stke}{\stack{E}}

\newcommand{\stks}{\stack{S}}
\newcommand{\stkt}{\stack{T}}

\newcommand{\stkMod}[1][]{\stack{Mod}_{#1}}

\newcommand{\stkAut}[1][]{\stack{Aut}_{#1}}

\newcommand{\stkGr}{\stack{Gr}}

\newcommand{\astk}[1]{#1^+}

\newcommand{\approxto}[1][]{\xrightarrow[#1]{\approx}}

\begin{document}

\title[On quantizations of complex contact manifolds]
{On quantizations of complex contact manifolds}

\author[P. Polesello]{Pietro Polesello}
\address{Dipartimento di Matematica\\ 
Universit\`a di Padova\\ 
via Trieste 63\\
35121 Padova, Italy}
\email{pietro@math.unipd.it}

\thanks{The author is a member of GNAMPA of the Istituto Nazionale di Alta Matematica (INdAM). Research partially supported by Erwin Schr\"odinger Institute's JRF grant, by Padova University's grant CPDA122824/12 and by Fondazione Cariparo through the project \emph{Differential methods in Arithmetic, Geometry and Algebra}.}

\date{}


\begin{abstract}
A (holomorphic) quantization of a complex contact manifold is a filtered algebroid stack which is locally equivalent to the ring $\E$ of microdifferential operators and which has trivial graded. The existence of a canonical quantization has been proved by Kashiwara. In this paper we consider the classification problem, showing that the above quantizations are classified by the first cohomology group with values in a certain sheaf of homogeneous forms. Secondly, we consider the problem of existence and classification for quantizations given by algebras.
\end{abstract}

\maketitle

\section*{Introduction}
This paper deals with the problem of existence and classification for the quantizations of complex contact manifolds. 

Precisely, consider a complex contact manifold $(Y, \alpha)$, \emph{i.e.} a 
complex manifold $Y$ endowed with a projective 1-form $\alpha$ (a global section of the projective cotangent bundle $P^*Y \to Y$) such that $d\alpha|_H$ is non-degenerate,  where $H$ denotes the codimension 1 sub-bundle of the tangent bundle associated to $\alpha$ via projective duality.
Let $\shl= \alpha^*\O_{P^*Y}(1)$ be the pull-back of the relative Serre sheaf on $P^*Y$: it is a locally free $\O_Y$-module of rank 1 endowed with a Lie bracket induced by $\alpha$ (the Lagrange bracket). Finally, denote by $\shs_Y$ the sheaf on $Y$ whose local sections are symbols, \emph{i.e.} series $\sum_{j=-\infty}^m f_j$ with  $f_j\in\shl^{\tens j}$
and satisfying some growth conditions. $\shs_Y$ is endowed with a natural filtration and the associated graded sheaf is the "homogeneous coordinates ring" $\O_Y^{hom}=\DSum_{m\in\Z} \shl^{\tens m}$.

A (holomorphic) quantization algebra on $(Y,\alpha)$ is a sheaf of algebras $\sha$ on $Y$ locally isomorphic to $\shs_Y$ as $\C_Y$-modules, in such a way that the product on $\shs_Y$ induced by that of $\sha$ is a formal bidifferential operator and is compatible both with the algebra structure on $\O_{Y}$ and with the Lie bracket on $\shl$. Such $\sha$ is naturally filtered with $\O_Y^{hom}$ as graded sheaf and the above compatibility conditions mean that $\Gr (\sha)\simeq\O_Y^{hom}$ is a graded algebra isomorphism commuting with the Poisson structures (induced by the commutator on the left-hand side and by the Lagrange bracket on the right-hand side).

Darboux's theorem for complex contact manifolds asserts that for any point $y\in Y$ there exists an open neighborhood $V$ of $y$ and a contact transformation $i\colon V \to P^*M$ for some complex manifold $M$ (here  $P^*M$ is endowed with the canonical contact structure given by the Liouville 1-form). Since any quantization algebra on $V$ is locally isomorphic through $i$ to the $\C$-algebra $\E_{P^*M}$ of Sato's microdifferential operators (a localization of the algebra $\shd_M$ of differential operators on $M$), it follows that the quantization algebras on $Y$ are nothing but $\E$-algebras, \emph{i.e.} $\C$-algebras locally isomorphic to $\opb i \E_{P^*M}$ for any Darboux chart $i\colon Y\supset V \to P^*M$.

In the case $Y=P^*M$, formal (\emph{i.e.} without growth conditions) $\E$-algebras were classified by Boutet de Monvel in~\cite{BoutetdeMonvel1999}.

For a complex contact manifold $Y$, the situation is more complicated, since $\E$-algebras may not exist globally. However Kashiwara in~\cite{Kashiwara1996} proved that there exists a canonical stack (sheaf of categories) of modules over locally defined $\E$-algebras. In fact, this stack is equivalent to the stack of modules over an algebroid stack $\stke_Y$ (see~\cite{Kontsevich2001,D'Agnolo-Polesello2005}). Moreover, $\stke_Y$ has the same properties of an $\E$-algebra: it is filtered, locally equivalent to an $\E$-algebra and it has $\O_Y^{hom}$ as associated graded stack. 

Hence, it makes sense to say that $\stke_Y$ is a (holomorphic) quantization of $Y$ and to replace 
$\E$-algebras by $\C$-linear stacks locally equivalent to $\stke_Y$. Being locally modeled on $\E$-algebras, these objects have locally trivial graded, but their associated graded stack in general is non-trivial.
Thus, we are lead to define a $(\E,\sigma)$-algebroid, playing the role of quantization of $Y$, as a filtered $\C$-linear stacks which is locally equivalent to $\stke_Y$ and which has $\O_Y^{hom}$ as associated graded stack.

In this paper we classify all $(\E,\sigma)$-algebroids on any complex contact manifold $Y$. This is done by means of the cohomology group $H^1(Y;\Omega_Y^{1,cl}(0))$, where $\Omega_Y^{1,cl}(0)$ denotes the push-forward of the sheaf of closed $0$-homogeneous 1-forms on the canonical symplectification of 
$Y$. Furthermore, we go into the study of the existence and classification of $\E$-algebras on $Y$, thus generalizing the results of~\cite{Kashiwara1996}  and~\cite{BoutetdeMonvel1999}. Finally, we compare the classification of $\E$-algebras with that of $(\E,\sigma)$-algebroids, focusing mainly on the formal case.

Note that in~\cite{Polesello2007} it is proved that the canonical $(\E,\sigma)$-algebroid $\stke_Y$ is the unique one endowed with an anti-involution corresponding to the operation of taking the formal adjoint of microdifferential operators, and in~\cite{D'Agnolo-Polesello2011} we classify (both up to equivalence and to Morita equivalence) $\C$-linear stacks locally equivalent to $\stke_Y$ without any assumption on their graded stack (called $\E$-algebroids in~\cite{D'Agnolo-Polesello2011}).

The problem of the quantization of the homogeneous coordinates sheaf have been considered 
in~\cite{Kontsevich2001,Kaledin2006} in the projective Poisson/symplectic case. On real contact manifolds or, more generally, real conic Poisson manifolds of constant rank, quantization algebras have been considered by Boutet de Monvel under the name of Toeplitz algebras (see~\cite{BoutetdeMonvel1995,BoutetdeMonvel2002}).

\medskip

This paper is organized as follows: 
in Section~\ref{section:contact} we recall the main definitions and properties of complex 
contact manifolds.
In Section~\ref{section:DefQuant} we review the definition of quantization algebras.
In Section~\ref{section:E-alg} we recall the classification of the $\E$-algebras on $P^*M$.
In Section~\ref{section:class} we define the $(\E,\sigma)$-algebroids on a complex contact manifold and classify them (Theorem~\ref{th:class}); we then focus on compact K\"ahler manifolds.
In Section~\ref{section:vs} we investigate the existence of $\E$-algebras on complex contact manifolds and compare their classification with that of $(\E,\sigma)$-algebroids, mainly in the formal case. 

\medskip
\noindent
{\bf Acknowledgements}
The author had the occasion of visiting Paris VI University and ESI - University of Vienna during the preparation of this paper. Their hospitality is gratefully acknowledged. 

The author also wishes to thank Louis Boutet de Monvel, Andrea D'Agnolo and Matthieu Carette for several useful discussions, comments and insights. 

\medskip
\noindent
{\bf Notation and conventions}
All the filtrations are over $\Z$, increasing and exhaustive. 
If $\shm$ is a filtered sheaf, we will denote by $\Gr(\shm)$ its associated graded sheaf, and by 
$\Gr_i(\shm)$ the sub-sheaf of elements of degree $i$.
We will use a similar notation for the graded morphism associated to a morphism of filtered sheaves.

All the algebra laws are associative.
If $\sha$ is a sheaf of algebras, we will denote by $\sha^\times$ the sheaf of groups of its 
invertible elements and, for each section $a\in \sha^\times$, by $\ad(a)\colon \sha \to \sha$ the algebra isomorphism $b\mapsto aba^{-1}$. 

If $Y$ is a complex manifold, $\Omega_Y^i$ will denote the sheaf of holomorphic $i$-forms and 
$\O_Y$ (resp. $\Theta_Y$) that of holomorphic functions (resp. vector fields). We also set 
$\Omega_Y=\Omega^{\dim Y}_Y$, the canonical bundle of $Y$. For each section $\zeta\in \Theta_Y$, the Lie derivative (resp. the contraction) along $\zeta$ is denoted by $L_\zeta$ (resp. $i_{\zeta}$). These are related to the exterior differential $d$ by the Cartan formula $L_\zeta = di_\zeta + d i_\zeta$.

\section{Complex contact manifolds}\label{section:contact}

We recall here the definition and some basic properties of complex contact 
manifolds. References are made to~\cite[Section 3.3]{S-K-K} 
and~\cite{LeBrun-Salamon1994,Beauville2007}.

\subsection{Complex contact structures}\label{subsection:contact}
Let $Y$ be a complex manifold and $\pi\colon P^*Y\to Y$ its projective 
cotangent bundle (here $P^*Y = \dot T^*Y/\C^{\times}$, where $\dot T^*Y$ denotes the cotangent bundle of $Y$ with the zero-section removed).
Recall that a projective 1-form on $Y$ is a global section of $\pi$. By projective duality, this corresponds to a hyperplane distribution, \emph{i.e.} a  codimension 1 sub-bundle of the tangent bundle $TY$.

Given a projective 1-form $\alpha\colon Y\to P^*Y$, consider the associated cartesian diagram
$$
\xymatrix@C3em{
U= Y\times_{P^*Y}UY\ar[d]^-{\gamma} \ar[r] 
& UY \ar[d] \\ 
Y \ar[r]_-{\alpha} & P^*Y}
$$
where $T^*Y \times_Y P^*Y \supset UY \to P^*Y$ is the tautological (or universal) line bundle.

By composing the morphism $U\to UY$  with the natural projection $UY\to T^*Y$, we get an injective morphism of vector bundles (denoted by the same symbol) $\alpha\colon U\to T^*Y$. 
Hence, the line bundle $\gamma\colon U\to Y$ measures the failure to lift the projective 1-form 
$\alpha$ to a nowhere vanishing 1-form on $Y$. Moreover, the chain of vector bundle morphisms
$$
U \to[\Delta] U\times_Y U\to[id_U\times\alpha] U \times_Y T^*Y \to[^t d\gamma] T^*U
$$
(here $\Delta$ is the diagonal map and $^t d\gamma\colon U \times_Y T^*Y \to T^*U$ denotes the morphism induced by the transpose of the differential $d\gamma_v \colon T_vU \to T_{\gamma(v)}Y$ of $\gamma$ at $v\in U$) defines a 1-form $\tilde\alpha$ on $U$ which never vanishes outside the zero-section. 

Let $eu$ be the infinitesimal generator of the $\C^\times$-action on $U$; then $L_{eu}\tilde\alpha=\tilde\alpha$ (that is, $\tilde\alpha$ is $1$-homogeneous) and $i_{eu}\tilde\alpha=0$. 

Let $\shh$ denote the sheaf of sections of the hyperplane distribution
associated  to $\alpha$ and $\O_{P^*Y}(-1)$ the sheaf of sections of the  the tautological line bundle on $P^*Y$.  Let $\O_{P^*Y}(1)$ be its dual and set $\shl=\alpha^*\O_{P^*Y}(1)$. Then $\shl^{\tens -1}$ is identified with the sheaf of sections of $\gamma \colon U \to Y$ and we get dual exact sequences of locally free $\O_Y$-modules
\begin{equation}
\label{eq:alpha}
\xymatrix{0 \ar[r] & \shh \ar[r] & \Theta_Y \ar[r]^-{^t\alpha} & \shl \ar[r] &  0\\ 
0 \ar[r] & \shl^{\tens -1} \ar[r]^-{\alpha} & \Omega^1_Y \ar[r] & \shh^{\tens -1}\ar[r] & 0}
\end{equation}

The differential $d\alpha\colon \shl^{\tens -1}\to \Omega^2_Y$ is by definition the composition of 
$\alpha$ with the exterior differential. 
One checks that for any $k \in \N$ the morphism $\alpha\wedge (d\alpha)^k \colon \shl^{\tens -(k+1)}\to \Omega_Y^{2k+1}$ is $\O_Y$-linear, hence it defines a global section of $\Omega^{2k+1}_Y\tens[\O_Y]\shl^{\tens (k+1)}$.

\begin{definition-proposition}\label{def-prop:contact form}
In the above situation, $\alpha$ is called a contact form if one of the following equivalent conditions is satisfied:
\begin{itemize}
\item[i)] The $\shl$-valued skew form $\shh \tens_\C \shh \to \shl, \quad (v, w) \mapsto {}^t\alpha ([v, w])$ is everywhere non-degenerate (here $[\cdot,\cdot]$ denotes the Lie bracket on $\Theta_Y$).
\item[ii)] $\dim Y=2n + 1$ and $\alpha \wedge (d\alpha)^n$ is a nowhere vanishing section of $\Omega_Y\tens[\O_Y]\shl^{\tens (n+1)}$.
\item[iii)] $\tilde\omega=d\tilde\alpha$ is a symplectic form on $\widetilde Y = 
U \setminus \{\mbox{zero-section}\}$.
\end{itemize}
\end{definition-proposition}

Note that 
ii) implies that $\shl^{\tens -(n+1)}\simeq\Omega_Y$ and 
from the Cartan formula it follows that $L_{eu}\tilde\omega=\tilde\omega$ (that is, $\tilde\omega$ is 1-homogeneous).

\begin{definition}\label{def:contact}
\begin{itemize}
\item[i)] A complex contact manifold is a complex manifold $Y$ endowed with a contact form 
$\alpha$.
If $\alpha$ lifts to a 1-form, the contact manifold $(Y,\alpha)$ is called exact.
\item[ii)] A contact transformation $\varphi\colon (Y,\alpha) \to (X,\beta)$ between complex contact manifolds of the same dimension is a morphism of complex manifolds satisfying\footnote{Note that, in general the pull-back $\varphi^*(\beta)$ is defined only when $\{(y,\beta \circ \varphi(y)) ; \ y\in Y \}\cap P^*_Y X = \emptyset$, where $P^*_Y X$ denotes the projective conormal bundle of $Y$.} 
$\varphi^*(\beta)=\alpha$.
\end{itemize}
\end{definition}

Given a complex contact manifold $(Y,\alpha)$, we will refer to $\shl$ (resp. $\shh$) as the contact line bundle (resp. contact distribution), and to the principal $\C^\times$-bundle $\gamma\colon \widetilde Y \to Y$ together with the 2-form $\tilde\omega$ on $\widetilde Y$, as the canonical symplectification. Note that $(Y,\alpha)$ is exact if and only if $\shl$ is trivial or, equivalently, if $\gamma$ has a global section. Note also that any contact transformation $\varphi\colon Y \to X$ is a local biholomorphism and lifts to a homogeneous symplectic transformation $\tilde\varphi\colon \widetilde Y \to\widetilde X$. 

\begin{remark}\label{rk:1-dim} 
We include in the definition all $1$-dimensional complex manifolds as degenerate case.
Indeed, if $\dim Y =1$ then $Y$ is identified with $P^*Y$, hence it has a canonical contact form
(which is in fact unique), whose contact line bundle is the sheaf $\Theta_Y$ of vector fields and the contact distribution nothing but its zero-section. Note that $Y$ is exact if and only if it admits a nowhere vanishing 1-form, hence if and only if it is either non-compact or compact of genus 1.
\end{remark}

\begin{example}\label{ex:contact}
(i) The odd-dimensional complex projective space $\mathbb P^{2n+1}$ is a complex contact manifold, whose contact line bundle is $\O_{\mathbb P^{2n+1}}(2)$. Any 2-homogeneous symplectic form 
$\omega$ on $\C^{2n+2}$ defines a contact form on $\mathbb P^{2n+1}$ by pushing-down the 
1-form $\frac 12 i_{eu}\omega$, for $eu$ the Euler vector field. The canonical symplectification is thus 
$\C^{2n+2}\setminus \{0\}\to \mathbb P^{2n+1}$ with symplectic form $d (\frac 12 i_{eu}\omega) = \omega$. In homogeneous coordinates $[x_0,\dots,x_{2n+1}]$, the symplectic form 
$\omega = \sum_{i=0}^{n} dx_{2i}\wedge dx_{2i+1}$ defines the contact form
$\alpha=\frac 12 \sum_{i=0}^{n} (x_{2i}dx_{2i+1}-x_{2i+1}dx_{2i})$.

(ii) Let $M$ be a complex manifold. Then $P^*M$ is a complex contact manifold of dimension $2 \dim M+1$, with canonical contact form $\lambda$
obtained from the morphism
$$
^t\lambda\colon \Theta_{P^*M} \to[d\pi] \pi^*\Theta_M \to \O_{P^*M}(1) 
$$
where $d\pi$ is induced by the differential of $\pi\colon P^*M \to M$ and the right arrow by the relative Euler sequence.
The contact line bundle is $\O_{P^*M}(1)$ and
the canonical symplectification is thus
$\dot T^*M \to P^*M$, endowed with the symplectic structure induced by the Liouville form.

(iii) Let $J^1N=T^*N\times \C$ be the 1-jet bundle of an $n$-dimensional complex manifold $N$. 
Let $(t,\tau)$ be the system of homogeneous symplectic coordinates on $T^*\C$. Then $J^1 N$ is identified with the open subset of $P^*(N\times \C)$ defined by $\tau\neq 0$, hence it is endowed with the contact form induced by that of $P^*(N\times \C)$. 
Note that $J^1N$ is exact, a lift of the contact form being given by $\beta = \rho^*(\lambda) + dt$, where $\rho\colon J^1N\to T^*N$ denotes the projection and $\lambda$ the Liouville form on $T^*N$.
 
More generally, if $(X,\omega)$ is a complex symplectic manifold with $[\omega]=0$ in 
$H^2(X;\C_X)$, there exists a principal $\C$-bundle $\rho\colon Y \to X$ and a 1-form $\beta$ on $Y$ such that $d\beta=\rho^*(\omega)$ and $i_v(\beta)=1$, for $v$ the infinitesimal generator of the $\C$-action. 
Then $\beta$ descends to a contact form on $Y$. Such exact contact manifold is called a contactification of  $(X,\omega)$.
\end{example}

Recall that Darboux's theorem for complex contact manifolds asserts that a local model 
for a contact manifold $Y$ is an open subset of $P^*M$ for a complex manifold $M$ (see for example ~\cite[Section 3.3]{S-K-K}). More precisely, for any point $y\in Y$ there exist an open neighborhood $V$ of $y$, a complex manifold $M$ and a contact transformation $i\colon V \to P^*M$.
We call the pair $(V,i)$ a Darboux chart.

\subsection{Lagrange brackets and homogeneous Poisson structures}
Let $(Y,\alpha)$ be a complex contact manifold, with contact line bundle $\shl$, contact distribution 
$\shh$ and canonical symplectification  $\gamma\colon \widetilde Y\to Y$.

We denote by $\Theta_Y^c\subset \Theta_Y$ the sub-sheaf of contact vector fields, that is, of those
vector fields $v\in \Theta_Y$ such that $L_v$ preserves $\shh$ (\emph{i.e.} $L_v(w)=[v,w]\in\shh$ 
for every $w\in \shh$).
By the Jacobi identity, it follows that $\Theta_Y^c$ is closed under the Lie bracket, so that it inherits 
a Lie algebra structure. 
Note that a vector field is a contact vector field if and only if its local flow consists in contact transformations.

\begin{proposition}
The morphism $^t\alpha \colon \Theta_Y \to \shl$ in~\eqref{eq:alpha} restricts to an isomorphism of 
$\C_Y$-modules $\Theta_Y^c\isoto \shl$.
\end{proposition}

\begin{proof}
By the exact sequence~\eqref{eq:alpha}, $v\in\Theta_Y$ is a contact vector field if and only if 
$^t\alpha([v,w])=0$ for every $w\in \shh$. Since by Definition-Proposition~\ref{def-prop:contact form}~(i) the skew form $^t\alpha([\cdot,\cdot])$ on $\shh$ is non-degenerate, a contact vector field $v$ belongs to 
$\shh$ if and only if $v=0$. Hence $^t\alpha$ is injective on $\Theta_Y^c$.

Take $l\in\shl$. Since $^t\alpha \colon \Theta_Y \to \shl$ is surjective, locally there exists $u\in\Theta_Y$ such that $^t\alpha (u)=l$. Consider the map $\shh \to \shl, \quad w\mapsto {}^t\alpha([u,w])$. 
It is $\O_Y$-linear, hence there exists a unique section $u'\in \shh$ such that $^t\alpha([u,\cdot])= {}^t\alpha([u',\cdot])$ on $\shh$. It follows that $^t\alpha([u-u',w])=0$ for every $w\in \shh$, that is $u-u'\in \Theta_Y^c$, and $^t\alpha(u-u')=l$. 
\end{proof}
Denote by
$$
R\colon \shl \to \Theta_Y^c, \quad l\mapsto R_l
$$
the inverse of $^t\alpha|_{\Theta_Y^c}$ ($R_l$ is known as the Reeb field of $l$).
This defines a $\C$-linear splitting of the exact sequence~\eqref{eq:alpha}.

Let  $l,m\in\shl$ and set 
$$
\{l,m\}= {}^t\alpha ([R_l,R_m])
$$
Then $\{\cdot,\cdot\}$ defines a Lie bracket on $\shl$ (the Lagrange bracket), and $^t\alpha$ 
and $R$ thus become isomorphisms of sheaves of complex Lie algebras. Note that one has $\{l,fm\}=L_{R_l}(f)m + f\{l,m\}$ for any $f\in\O_Y$.

Set
$$
\O_Y^{hom}=\DSum_{m\in\Z} \shl^{\tens m}
$$ 
This is the graded algebra of "homogeneous coordinates", since there is an identification
$\O_Y^{hom}= \gamma_* \O_{(\widetilde Y)}$, where $\O_{(\widetilde Y)}\subset
\O_{\widetilde Y}$ is the sub-sheaf of holomorphic functions on $\widetilde Y$ which are rational along the fiber.
It has a natural Poisson algebra structure defined by extending the Lagrange bracket  $\{\cdot,\cdot\}$ on $\shl$ as a derivation on each argument. 
It is of degree $-1$ ({\em i.e.} $\deg \{l,m\} = \deg l + \deg m -1$) and coincides with the Poisson structure induced by that on $\O_{\widetilde Y}$ associated to the symplectic form $\widetilde \omega$. In particular, one has $\{l,f\}=L_{R_l}(f)$ for any $l\in\shl$ and $f\in\O_Y$. 

\begin{notation}
For any $k\in \Z$, we denote by $Der^P_{\C}(\O_Y^{hom})(k)$ the sheaf of $\C$-derivations of $\O_Y^{hom}$ which are graded of degree $k$ and preserve the Poisson structure\footnote{Recall that a derivation $\delta\colon \O_Y^{hom}\to \O_Y^{hom}$ is graded of degree $k$ if $\delta(\shl^{\tens m})\subset \shl^{\tens m+ k}$ for any $m\in\Z$, and it preserves the Poisson structure if $\delta(\{f,g\})=\{\delta(f),g\}+\{f,\delta(g)\}$ for any $f,g\in  \O_Y^{hom}$.}, by $\Theta^s_{\widetilde Y}(k)$ the sheaf of $k$-homogeneous symplectic vector fields\footnote{Recall that a vector field $v$ on $({\widetilde Y},\tilde\omega)$ is symplectic if $L_{v}\tilde\omega=0$ and it is $k$-homogeneous if $L_{eu}v=[eu,v]=kv$.} on $\widetilde Y$ and by $\Omega_{\widetilde Y}^i(k)$ that of $k$-homogeneous $i$-forms on 
$\widetilde Y$ ({\em i.e.} those satisfying $L_{eu}(\eta)=k\eta$). We also set
$$
\Theta^s_Y(k)=\gamma_*\Theta^s_{\widetilde Y}(k), \quad 
\Omega_Y^i(k)=\gamma_*\Omega_{\widetilde Y}^i(k)
$$ 
\end{notation}

Note that $\Theta^s_Y(0)=\Theta^c_Y$ and there are isomorphisms of $\C_Y$-modules 
\begin{equation}
\label{pr:GradDer}
H\colon \Omega_Y^{1,cl}(k+1)\isoto \Theta^s_{Y}(k), \quad 
L\colon\Theta^s_Y(k)\isoto Der^P_{\C}(\O_Y^{hom})(k)
\end{equation}
induced by the the Hamiltonian isomorphism on $\widetilde Y$ and by the Lie derivative, respectively
(here  the upper index $cl$ denotes closed forms).

By identifying $\shl^{\tens k}$ with $\Omega^0_Y(k)$ and by using the Cartan formula,
one gets for any $k\neq 0$ an isomorphism of $\C_Y$-modules
$$
d\colon\shl^{\tens k}\isoto\Omega_Y^{1,cl}(k)
$$
In particular, for $k=1$, it follows that the contact vector field $R_l$ associated to $l\in \shl$ is obtained by pushing down to $Y$ the symplectic vector field $H_{dl}$ on $\widetilde Y$. 

For $k = 0$ one has $0 = d i_{eu}\eta + i_{eu} d\eta$, hence there is an exact sequence
$$
0 \to \C_Y \to \O_Y \to[d] \Omega_Y^{1,cl}(0) \to[i_{eu}]\C_Y \to 0
$$
Together with the exact sequence
\begin{equation}\label{eq:deRham}
0 \to \C_Y \to \O_Y \to[d] \Omega_Y^{1,cl} \to 0
\end{equation} 
it gives rise to the exact sequence
\begin{equation}\label{eq:OmegaY}
0 \to \Omega_Y^{1,cl} \to[\gamma^*] \Omega_{Y}^{1,cl}(0) \to[i_{eu}] \C_Y \to 0
\end{equation}
(here $\gamma^*$ denotes the pull-back via $\gamma \colon \widetilde Y\to Y$).

\begin{lemma}\label{lemma:Atiyah}
The class of the extension of $\C_Y$-modules~\eqref{eq:OmegaY} is the Atiyah class\footnote{Recall that the Atiyah map $a\colon H^1(Y;\O_Y^\times)\to H^1(Y;\Omega_Y^1)$ induced by $d\log$ factors through $H^1(Y;\Omega_Y^{1,cl})$.} $-a(\shl)\in H^1(Y;\Omega_Y^{1,cl})$.
\end{lemma}

\begin{proof}
First, note that~\eqref{eq:OmegaY} is a sub-sequence of the exact sequence
$$
0 \to \Omega_Y^{1} \to[\gamma^*] \Omega_{Y}^{1}(0) \to[i_{eu}] \O_Y \to 0
$$
After tensoring it with $\shl$, this becomes isomorphic to the first jet sequence of $\shl$, whose extension class\footnote{Here we use Atiyah's original definition, which differs by a sign to that in more recent literature, as for example in~\cite{Kapranov1999}.} is $-a(\shl)\in Ext^1_{\O_Y}(\shl,\Omega_Y^1\tens[\O_Y]\shl)\simeq H^1(Y;\Omega_Y^1)$. Recall that this may be realized as the class of the $\Omega^1$-torsor of the connections on $\shl$. Then, as an element of $H^1(Y;\Omega_Y^{1,cl})$, it is realized as the $\Omega_Y^{1,cl}$-torsor of the flat connections on $\shl$.

The result then follows, since any splitting of the sequence~\eqref{eq:OmegaY} is given by a global section $\beta\in \Omega_Y^{1,cl}(0)$ satisfying $i_{eu}\beta=1$, \emph{i.e.} such that $d + \beta\wedge\cdot$ descends to a flat connection on $\shl$.
\end{proof}

\section{Quantization algebras}\label{section:DefQuant}

We review here the definition of a quantization algebra on a complex contact manifold (called star algebras in~\cite{BoutetdeMonvel1999}).

Let $Y$ be a complex contact manifold, $\gamma\colon \widetilde Y \to Y$ its canonical symplectification and $\shl$ the associated contact line bundle.

For $m\in \Z$, let 
$$
F_m\widehat\shs_Y= \prod_{j=-\infty}^m\shl^{\tens j}
$$ 
be the sheaf of formal symbols of degree $\leq m$, \emph{i.e.} series $\sum_{j=-\infty}^m f_j$ with  $f_j\in\shl^{\tens j}$, and set 
$$
\widehat\shs_Y=\Union_{m\in\Z}F_m\widehat\shs_Y
$$
We denote by $\shs_Y\subset\widehat\shs_Y$ the sub-sheaf of (holomorphic) symbols, \emph{i.e.} symbols which, on an open subset $U\subset Y$ where they are defined, are subject to the estimates
\begin{equation}\label{eq:estmicrod}
     \left\{ \begin{array}{l}
     \mbox{for any compact subset $K \subset\opb \gamma (U)$ there exists a constant}\\
     \mbox{$C_K>0$ such that }
     \sup\limits_{K}\vert f_{j}\vert \leq C_K^{-j}(-j)! \mbox{ for all $j<0$.}
     \end{array}\right.
\end{equation}
(Here and in the sequel, we shall identify $\shl^{\tens j}$ with the sub-sheaf of $\oim \gamma \O_{\widetilde Y}$ of $j$-homogeneous functions.)

The sheaf $\shs_Y$ inherits a filtration from that of $\widehat\shs_Y$ and one has
$\plim[m\in \N] \shs_Y/F_{-m}\shs_Y \simeq \widehat \shs_Y$. Hence
$$
\Gr(\shs_Y)=\Gr(\widehat\shs_Y)=\O^{hom}_Y
$$

\begin{remark}
Since the $\O_Y$-module $\shl$ is locally free of rank $1$, the algebra $F_0\widehat\shs_Y$ is nothing but the completion of the symmetric algebra of $\shl^{\tens -1}$. Recall from Section~\ref{subsection:contact} that $U \to Y$ denotes the line bundle whose sheaf of sections is 
$\shl^{\tens -1}$ and let $U^* \to Y$ be its dual. Then $F_0\widehat\shs_Y$ is identified with the sheaf 
$\O_{U^*}\hat |_Y$, the formal completion of $\O_{U^*}$ along $Y$ (regarded as the zero section),
and its sections are thus the holomorphic functions in the formal neighborhood of $Y$ in $U^*$. 

Recall that the Borel-Laplace transform defines an isomorphism of $\C_Y$-modules
$ \O_{U^*}\hat |_Y \isoto   \O_{U}\hat |_Y$.
Via the previous identification, we get an isomorphism of $\C_Y$-modules 
$$
B\colon F_0\widehat\shs_Y \isoto   \O_{U}\hat |_Y
$$ 
In local coordinates $(y;\tau)\in \widetilde Y = U \setminus \{\mbox{zero-section}\}$,  where $\tau$ is the linear coordinate on the fiber, a section of $F_0\widehat\shs_Y$ is written as $ f(y;\tau)  =\sum_{j\leq0}  f_j(y) \tau^j$ with $f_j\in\O_Y$ and
$$
B(f)=\frac{1}{2\pi i}\int_\gamma  f(y;\tau) \frac{e^{t\tau}}{\tau} d\tau=\sum_{j\geq0}  f_{-j}(y) \frac{t^j}{j!}
$$
for $\gamma$ a counterclockwise oriented circle around 0 in $\C$. 
Then $B$ restricts to an isomorphism $F_0\shs_Y\isoto\O_{U}|_Y$, since by~\eqref{eq:estmicrod} for any  compact subset $K\subset \opb \gamma (U)$ the series $\sum_{j\geq0}  \sup\limits_{K}\vert f_{-j}\vert \frac{C_K^j}{j!}$ converges.
\end{remark}

For any $k\in \Z$, let $\D_{\widetilde Y}(k)=\{D\in \D_{\widetilde Y}; \, [L_{eu},D]=kD\}$ denote the sheaf of $k$-homogeneous linear differential operators on $\widetilde Y$, and set 
$$
\D_{Y}(k)=\gamma_*\D_{\widetilde Y}(k)
$$
Note that any operator $D\in \D_Y(k)$ defines a $\C_Y$-linear morphism $\shl^{\tens j}\to \shl^{\tens j+k}$ for any $j\in\Z$. In local coordinates $(y;\tau)\in \widetilde Y$, where $\tau$ is the linear coordinate on the fiber, $D$ is written as a finite sum 
$$
\tau^k\sum_{i\in \N, \alpha\in \N^{\dim Y}}  a_{i,\alpha}(y) (\tau\partial_\tau)^i\partial_y^\alpha
$$ 
with  $a_{i,\alpha}\in\O_Y$.

We denote by $\shd_{\shs_Y}$ the sheaf of rings of homogeneous differential operators on $\shs_Y$. 
Its sections are $\C$-linear morphism $D\colon \shs_Y \to \shs_Y$ which can be written as $D= \sum_{j=-\infty}^m D_j$ with $D_j\in \D_{Y}(j)$.

An algebra law $\star$ on $\shs_Y$ is said differential (resp. unitary) if  for any $f\in \shs_Y$ both $f \star \cdot$ and $\cdot \star f$ belong to $\shd_{\shs_Y}$ (resp. $f \star 1 = f = 1\star f$ with 
$1\in \O_Y$). This amounts to say that $\star = \sum_{j=-\infty}^m B_j$ with 
$B_j\in \DSum_{i+k=j} \D_{Y}(i)\tens[\O_Y]\D_Y(k)$ for any $j\leq m$\footnote{In fact, since $\star$ is associative, $B_m$ is necessarily the multiplication by some $b\in \shl^{\tens m}$.}
(resp. $B_0(\cdot, 1)=1=B_0(1,\cdot)$ and $B_j(\cdot,1)=0=B_j(1,\cdot)$ for $j \neq 0$).

If an algebra law $\star = \sum_{j\leq m} B_j$ on $\shs_Y$ is differential and unitary, then $B_0$ is the point-wise multiplication and $B_j$ has no term of order $0$ for $j \neq 0$ (hence it vanishes for $j>0$). 
In particular, $(\shs_Y,\star)$ is a filtered $\C$-algebra and there is an identification of graded algebras $\Gr(\shs_Y,\star)=\O_Y^{hom}$.

\begin{lemma}\label{lemma:diff}
Let $\star_1,\star_2$ be differential and unitary algebra laws on $\shs_Y$.
Then any filtered $\C$-algebra morphism $D\colon (\shs_Y,\star_1) \to (\shs_Y,\star_2)$ is given by 
$D=1+ \sum_{i=-\infty}^{-1} D_i\in \D_{\shs_Y}$ with $D_i$ having no term of order $0$.
\end{lemma}

\begin{proof}
For any $f \in \O_Y\subset F_0\shs_Y$, write $D(f) =  \sum_{i=-\infty}^0D_i(f)$ with $D_i(f)\in \shl^{\tens i}$. Then $D_0$ defines a $\C$-algebra endomorphism of $\O_Y$, which is necessarily the identity (this is well-known, see for example~\cite[Lemma 2.1.15]{Kashiwara-Schapira2013} for a proof). Hence $D = 1+ \sum_{i\leq-1} D_i$.

Write $\star_l = \sum_{j\leq 0}B^l_j$ with $B^l_j\in \DSum_{i+k=j} \D_{Y}(i)\tens[\O_Y]\D_Y(k)$ for $l=1,2$.
For any $f \in \O_Y$ we have $D(f\star_1\cdot) = D(f)\star_2 D(\cdot)$, hence
$$
\sum_{i + j =  m} D_i(B^1_j(f,\cdot)) = \sum_{i+j+k=m} B^2_j(D_i(f),D_k(\cdot)) \quad \text{for any } m\leq 0.
$$
Suppose that $D_i\in \D_{Y}(i)$ for all $m < i <0$. 
Then
$$
D_m(B^1_0(f,\cdot)) - B^2_0(D_m(f),\cdot) - B^2_0(f,D_m(\cdot)) \in \D_{Y}(m)
$$
Since both $B^1_0(\cdot,\cdot)$ and $B^2_0(\cdot,\cdot)$ are the point-wise multiplication, we get 
$$
[D_m,  f] - D_m(f) \in \D_{Y}(m)
$$
that is, $[D_m,  f]  \in \D_{Y}(m)$ for any $f\in \O_Y$. It follows (see for example~\cite[Lemma 2.2.4]{Kashiwara-Schapira2013}) that $D_m \in \D_{Y}(m)$. By induction, we get that $D_i\in \D_{Y}(i)$ for any $i<0$. 

Moreover, from $D(1)=1$ it follows that $D_i(1)=0$ for any $i<0$.
\end{proof}

\begin{definition}
A (holomorphic) quantization algebra on $Y$ is a sheaf of $\C$-algebras $\sha$ on $Y$ 
which is locally isomorphic to $\shs_Y$ as $\C_Y$-modules, in such a way that the
algebra law on $\shs_Y$ induced by that of $\sha$ is differential and unitary, and the following diagram commutes
\begin{equation}
\label{eq:Jacobi}
\xymatrix@C3em{ F_1\sha\times F_1\sha \ar[r]^-{[\cdot,\cdot]} 
\ar[d]^-{\nu_1\times\nu_1} & 
F_1\sha\ar[d]^-{\nu_1} \\
\shl \times \shl\ar[r]^-{\{\cdot,\cdot\}} &
\shl}
\end{equation}
(Here the filtration $\{ F_m\sha\}$ on $\sha$ is induced by that on $\shs_Y$, $\nu_1$ denotes the symbol map of degree $1$, $[\cdot,\cdot]$ the commutator and $\{\cdot,\cdot\}$ the Lagrange bracket defined by the contact form).
\end{definition}
Note that Lemma~\ref{lemma:diff} ensures that there is a canonical graded algebra isomorphism 
$$
\nu\colon\Gr(\sha)\isoto\O_Y^{hom}
$$
hence the symbol map in~\eqref{eq:Jacobi} is well-defined. Moreover, the diagram~\eqref{eq:Jacobi} commutes if and only if $\nu$ preserves the Poisson structure on $\Gr(\sha)$ induced by the commutator $[\cdot,\cdot]$.

\begin{remark}\label{rk:formal}
By replacing $\shs_Y$ with $\widehat\shs_Y$ in the definition above, we get the notion of formal quantization algebra. Note that, if $\sha$ is a quantization algebra on $Y$, then $\widehat \sha = 
\plim[m\in \N] \sha/F_{-m}\sha$ is a formal quantization algebra and $\sha\subset \widehat \sha$.
\end{remark}

\section{Microdifferential algebras on $P^*M$}\label{section:E-alg}

Since a local model for a contact manifold $Y$ is an open subset of $P^*M$ for a complex manifold $M$, we start here by considering the quantization algebras on $P^*M$, following~\cite{BoutetdeMonvel1999}. The sheaf of Sato's microdifferential operators provides the local model for such algebras.

\subsection{$\E$-algebras}
Let $M$ be an $n$-dimensional complex manifold and $\pi\colon P^*M \to M$ its projective cotangent bundle. Recall from Example~\ref{ex:contact} (ii) that $P^*M$ is endowed with a canonical contact structure whose contact line bundle is $\O_{P^*M}(1)$ and whose canonical symplectification is the 
$\C^\times$-principal bundle $\gamma\colon \dot T^*M \to P^*M$.

Let $\E_{P^*M}$ be the sheaf on $P^*M$ of microdifferential operators (see~\cite{S-K-K,Kashiwara1986,Kashiwara2003}, and also~\cite{Schapira} for an exposition). Recall that 
$\E_{P^*M}$ is a central $\C$-algebra filtered by the order of the operators. In a local coordinate system $(x)$ on $M$ with associated local coordinates $(x;[\xi])$ on $P^*M$, a microdifferential operator $P$ of order $\leq m$ has a total symbol
$$
\sigma_{\mathrm{tot}}(P)\in F_m\shs_{P^*M}
$$
and the product structure is given by the Leibniz formula: 
$$
\sigma_{\mathrm{tot}}(PQ) = \sum_{\alpha\in\N^n} \frac{1}
{\alpha !} \partial^{\alpha}_{\xi}\sigma_{\mathrm{tot}}(P)
\partial^{\alpha}_x\sigma_{\mathrm{tot}}(Q)
$$
for $Q$ a microdifferential operator of total symbol $\sigma_{\mathrm{tot}}(Q)$. It follows that 
$\sigma_{\mathrm{tot}}$ locally induces a differential and unitary law $\star = \sum_{j\leq 0} B_j$ on 
$\shs_{P^*M}$ by setting $B_j = \sum_{|\alpha|=-j} \frac{1}
{\alpha !} \partial^{\alpha}_{\xi} \tens \partial^{\alpha}_x$ for $j\leq 0$.

Denote by $F_m\E_{P^*M}$ the sub-sheaf of operators of order $\leq m$ and by
$$
\sigma_m \colon F_m\E_{P^*M}\to \Gr_m\E_{P^*M}\simeq \O_{P^*M}(m)
$$ 
the symbol map of order $m$.
If $\sigma_m(P)$ is not identically zero, then one says that $P$ has order $m$.
In particular, $P\in F_m\E_{P^*M}$ is invertible if and only if $\sigma_m(P)$ is nowhere vanishing.

The symbol maps induce an isomorphism of graded algebras
\begin{equation}\label{eq:sigma}
\sigma\colon\Gr(\E_{P^*M})\isoto\O^{hom}_{P^*M}
\end{equation}
For any $P\in F_m\E_{P^*M}$ and $Q\in F_n\E_{P^*M}$, one has
$$
\sigma_{m+n-1}[P,Q]=\{\sigma_m(P),\sigma_n(Q)\}
$$
Therefore the isomorphism~\eqref{eq:sigma} preserves the Poisson structure and 
$\E_{P^*M}$ is a quantization algebra on $P^*M$. 
Moreover, the $\C$-algebra $\widehat\E_{P^*M}=\plim[m\in \N] \E_{P^*M}/F_{-m}\E_{P^*M}$ 
(\emph{i.e.} obtained by dropping the growth condition~\eqref{eq:estmicrod}) is a formal quantization algebra.

\begin{lemma}[cf. {\cite[Chap. II \S 3.2]{S-K-K}}]\label{lemma:E-morphism}
Let $\varphi\colon \E_{P^*M} \to\E_{P^*M}$ be a $\C$-algebra endomorphism. Then 
$\varphi$ is filtered with $\Gr(\varphi)=\id$.
\end{lemma}
See~\cite[Lemma 4.2.1]{D'Agnolo-Polesello2011} for an elementary proof.
This suggests the following
\begin{definition}\label{definition:E-alg}
An $\E$-algebra on $P^*M$ is a $\C$-algebra locally isomorphic to $\E_{P^*M}$ as a $\C$-algebra.
\end{definition}
By Lemma~\ref{lemma:E-morphism}, any $\E$-algebra $\sha$ comes equipped with a canonical filtration $\{ F_m\sha\}$ and with an isomorphism of graded algebras
\begin{equation}\label{eq:symbol}
\nu\colon\Gr(\sha)\isoto\O^{hom}_{P^*M}
\end{equation}
Also, any $\C$-linear isomorphism of $\E$-algebras $\varphi\colon \sha_1 \to \sha_2$ is filtered and compatible with the associated graded isomorphisms $\nu^1,\nu^2$. This means that  
$\varphi$ maps $F_m\sha_1$ to $F_m\sha_2$ in such a way that 
$\nu^2_m(\varphi(P)) = \nu^1_m(P)$ for all $P\in F_m\sha_1$. (Here  $\nu^i_m$ denotes the symbol map $F_m\sha_i\to F_m \sha_i/F_{m-1}\sha_i\simeq \O_{P^*M}(m)$ of order $m$,
for $i=1,2$.)

Given an $\E$-algebra $\sha$, it follows that the associated graded isomorphism~\eqref{eq:symbol} preserves the Poisson structure, so that $\sha$ is a quantization algebra.
Moreover, the converse is also true:

\begin{proposition}[cf. {\cite[Proposition 2]{BoutetdeMonvel1999}}]\label{pr:Q-alg}
Any quantization algebra on $P^*M$ is an $\E$-algebra.
\end{proposition}

This allows us to identify quantization algebras on $P^*M$ with $\E$-algebras. 

\begin{example}\label{ex:E-alg}
(i) Let $\shp$ be a locally free right $\E_{P^*M}$-module of rank $1$. Then 
$\shendo[\E_{P^*M}^\op-mod](\shp)$ is an $\E$-algebra. If $\shp = \opb{\pi}\shn\tens_{\opb{\pi}\O_M}\E_{P^*M}$ for a line bundle $\shn$ on $M$, the above $\E$-algebra may be written as
$$
\E_\shn = \opb {\pi}\shn\tens_{\opb{\pi}\O_M}\E_{P^*M} \tens_{\opb{\pi}\O_M}
 \opb{\pi}\shn^{\tens -1}
$$ 
In particular, $\E_{\Omega_M}\isofrom\E_{P^*M}^\op$ by the formal adjoint (see~\cite[Theorem 7.7]{Kashiwara2003}). More generally, we may replace $\shp$ and $\shn$ above by
$\C^\times$-twisted modules (see~\cite{Kashiwara1989}, and~\cite{D'Agnolo-Polesello2004} for a review on twisted modules).

(ii) Let $f\colon P^*M \to P^*N$ be a contact transformation. 
By a result of~\cite{S-K-K} (see also~\cite{Kashiwara1986,Kashiwara2003}), there exists an 
invertible\footnote{Recall that an $\sha^\op\tens[\C]\shb$-module $\shp$ is invertible if the functor 
$\shp\tens[\sha](\cdot) \colon \Mod(\sha) \to \Mod(\shb)$ is a $\C$-equivalence.} 
$\opb f \E_{P^*N}\tens[\C]\E_{P^*M}$-module $\shm_f$ which is locally free of rank $1$ as $\E_{P^*M}$-module. Then $\shp= \opb{\pi}\Omega_M\tens_{\opb{\pi}\O_M}\shm_f$ is an invertible 
$\opb f \E_ {P^*N}\tens[\C]\E_{P^*M}^\op$-module locally free of rank $1$ as right 
$\E_{P^*M}$-module and $\opb f \E_ {P^*N}\isoto \shendo[\E_{P^*M}^\op-mod](\shp)$. 
It follows that $\opb f \E_ {P^*N}$ is an $\E$-algebra.
\end{example}

\subsection{Classification}

We denote by
$$
\{\E \text{-algebras}\}_{P^*M}
$$
the set of isomorphism classes of $\E$-algebras on $P^*M$,  pointed by the class of $\E_{P^*M}$, and by $\shaut[\C-alg](\E_{P^*M})$ the sheaf of $\C$-algebra automorphisms of $\E_{P^*M}$.
Then there is an isomorphism of pointed sets
$$
\{\E \text{-algebras}\}_{P^*M}\simeq H^1(P^*M; \shaut[\C-alg](\E_{P^*M}))
$$
where the right-hand side is defined by \v Cech cohomology.

Let $\varphi\colon \E_{P^*M} \to\E_{P^*M}$ be a $\C$-algebra automorphism. 
By Lemma~\ref{lemma:E-morphism}, it is filtered with $\Gr(\varphi)=\id$. Therefore the assignment 
$$
F_m \E_{P^*M} \ni P \mapsto \sigma_{m-1}(\varphi(P)-P) \in \O_{P^*M}(m-1)
$$
defines a $\C$-derivation of $\O_{P^*M}^{hom}$ which is graded of degree $-1$. It is easily seen to 
preserve the Poisson structure, so that it is a section of $Der^P_{\C}(\O_{P^*M}^{hom})(-1)$, and we
denote by $s(\varphi)$ the section of $\Omega_{P^*M}^{1,cl}(0)$ obtained  by applying the isomorphisms~\eqref{pr:GradDer}. 

Since $s(\varphi \circ \varphi')=s(\varphi) + s(\varphi')$, we get a group morphism 
\begin{equation}\label{eq:key}
s\colon\shaut[\C-alg](\E_{P^*M}) \to \Omega_{P^*M}^{1,cl}(0)
\end{equation}
Note that, if $P\in F_m\E_{P^*M}^\times$, then $s(\ad(P))=d\log \sigma_m(P)$.
Set
$$
\E_{P^*M, 1} = \{P\in F_0\E_{P^*M}; \mbox{ $\sigma_0(P)=1$}\}\subset F_0\E_{P^*M} ^\times
$$

\begin{lemma}[cf. {\cite[Proposition 3]{BoutetdeMonvel1999}}]\label{lemma:key}
There is an exact sequence of sheaves of groups on $P^*M$
\begin{equation*}\label{eq:key}
1 \to \E_{P^*M, 1} \to[\ad]  \shaut[\C-alg](\E_{P^*M}) \to[s] \Omega_{P^*M}^{1,cl}(0)\to 0
\end{equation*}
\end{lemma}

\begin{proof} 
Let $\varphi\in\shaut[\C-alg](\E_{P^*M})$ and denote by $\tilde \varphi$ the induced $\C$-algebra automorphism of $\E_{\dot T^*M}= \opb \gamma \E_{P^*M}$. By~\cite[Chap. II \S 3.2]{S-K-K} (see also~\cite[Proposition 4.2.2]{D'Agnolo-Polesello2011} for a short proof), $\tilde \varphi$ is locally of the form $\ad(P)$ for some $\lambda\in \C$ and some invertible operator $P$ of order 
$\lambda$. 
Let $\O_{\dot T^*M}(\lambda)\subset \O_{\dot T^*M}$ be the sub-sheaf of $\lambda$-homogeneous functions, $F_\lambda\E_{\dot T^*M}$ the sheaf of operators of order $\lambda + \Z_{\leq 0}$ and $\sigma_\lambda \colon F_\lambda\E_{\dot T^*M}\to \O_{\dot T^*M}(\lambda)$ the symbol map of order $\lambda$. Then $s(\varphi) = d\log \sigma_\lambda(P)$, hence $s(\varphi)=0$ if and only if $\sigma_\lambda(P)$ is locally constant, \emph{i.e.} $\lambda=0$ and $\sigma_0(P)\in \C^\times$. In that case, locally on $P^*M$ one has $\varphi=\ad(P')$, for $P'=P\opb{\sigma_0(P)}\in \E_{P^*M, 1}.$

Let $\eta\in \Omega_{P^*M}^{1,cl}(0)$ and set $\lambda=i_{eu}(\eta)\in \C$. Then, locally on 
$\dot T^*M$ we may find an invertible operator $P$ of order $\lambda$ such that $\eta = d\log \sigma_\lambda(P)$. Then $\ad(P)$ descends to a locally defined $\C$-algebra automorphism 
$\varphi\colon\E_{P^*M} \to \E_{P^*M}$ satisfying $s(\varphi)=\eta$.
\end{proof}

Consider the associated long non-abelian exact sequence
\begin{equation}\label{eq:longexact}
H^1(P^*M;\E_{P^*M,1}) \to 
H^1(P^*M;\shaut[\C-alg](\E_{P^*M})) \to H^1(P^*M;\Omega_{P^*M}^{1,cl}(0))
\end{equation}
The pointed set  $H^1(P^*M;\E_{P^*M,1})$ classifies isomorphism classes of locally 
free right $\E_{P^*M}$-module of rank $1$ which are filtered with trivial graded
(\emph{i.e.} isomorphic to $\O^{hom}_{P^*M}$) and the first map in the sequence is described by 
$[\shm]\mapsto [\shendo[\E_{P^*M}^\op-mod](\shm)]$ (here $[\, \cdot \, ]$ denotes isomorphism class).

\medskip
 
Let $\shd_M$ denote the $\C$-algebra of linear differential operators on $M$. Recall that 
$\shd_M$ is identified to a filtered sub-algebra of $\oim\pi \E_{P^*M}$ and that 
$\Gr(\shd_M)\simeq \DSum_{m\in\N} \oim \pi \O_{P^*M}(m)$. One has $F_m\D_M = 0$ for $m<0$, $F_0\D_M = \O_M$ and for $m\geq 1$
\begin{equation}\label{eq:FD}
F_m\D_M   =   \{P\in \shd_M; [P,f]\in F_{m-1}\D_M \mbox{ for any } f\in \O_M\}
\end{equation}
From these, one gets
\begin{equation}\label{eq:O}
\O_M = \{P\in \shd_M; [P, \cdot] \text{ is nilpotent} \}
\end{equation}

\begin{lemma}\label{lemma:D-morphism}
Let $\psi\colon \D_M \to\D_M$ be a $\C$-algebra endomorphism. Then 
$\psi$ is filtered with $\Gr(\psi)=\id$.
\end{lemma}

\begin{proof}
Since $\psi$ preserves the commutator, by~\eqref{eq:O} it restricts to a $\C$-algebra endomorphism of $\O_M$. This needs to be the identity (see for example~\cite[Lemma 2.1.15]{Kashiwara-Schapira2013} for a proof), hence $\psi$ is a $\C$-algebra endomorphism of $\shd_M$ preserving symbols of order $0$. By induction and by using~\eqref{eq:FD}, one gets $\psi(P)-P\in  F_{m-1}\D_M$ for any $P\in  F_m\D_M$. Then $\psi$ preserves the filtration $\{ F_m\D_M \}$ and symbols of any order, \emph{i.e.} $\Gr(\psi)=\id$. 
\end{proof}

Denote by $\shaut[\C-alg](\shd_M)$ the group of $\C$-algebra automorphisms of $\shd_M$. 
Then the assignment $\psi \mapsto (\Theta_M \ni v \mapsto \psi(v) - v \in \O_M)$ defines an isomorphism of groups
\begin{equation}\label{eq:TDO} 
\shaut[\C-alg](\shd_M)\isoto\Omega_M^{1,cl}
\end{equation}
which sends $\ad(f)$ to $d\log f$ for any $f\in \O_M^\times$.
Since $\O_M^\times/\C^\times\isoto[d\log] \Omega_M^{1,cl}$ and $\O_M^\times=\shd_M^\times \subset \oim\pi\E_{P^*M}^\times$, it follows that the $\C$-algebra automorphisms of $\D_M$ are inner (\emph{i.e.} they are locally of the form $\ad(P)$ for some invertible operator $P$) and extend to inner automorphisms of $\E_{P^*M}$. We thus get a commutative diagram
\begin{equation}\label{eq:pi}
\xymatrix@C3em{
\oim\pi \shaut[\C-alg](\E_{P^*M}) \ar[r]^-{s} & \oim\pi\Omega_{P^*M}^{1,cl}(0)\\
\shaut[\C-alg](\shd_M) \ar[u] \ar[r]_-{\sim} & \Omega_{M}^{1,cl} \ar[u]_-{\pi^*} }
\end{equation}
where $\pi^*$ denotes the pull-back via $\pi\colon P^*M \to M$. 

Let $\dim M \geq 2$. Then for any $m<0$ one has $\oim\pi\O_{P^*M}(m)=0$, hence 
$\shd_M=\oim\pi \E_{P^*M}$ and the push-forward thus defines a morphism 
$\oim\pi \shaut[\C-alg](\E_{P^*M})\to \shaut[\C-alg](\shd_M)$. Moreover, one easily checks 
that the map $\pi^*$ in~\eqref{eq:pi} is an isomorphism and $\oim\pi \E_{P^*M,1}=1$. By Lemma~\ref{lemma:key}, it follows that all arrows in~\eqref{eq:pi} are isomorphisms.

From~\cite[Lemma 1]{BoutetdeMonvel1999}, one immediately gets
\begin{lemma}\label{lemma:R^1}
If $\dim M \geq 2$, then $R^1\oim\pi \Omega_{P^*M}^{1,cl}(0) =0$.
\end{lemma}
In particular, in the above situation we have
\begin{equation}\label{eq:R^1}
H^1(P^*M;\Omega_{P^*M}^{1,cl}(0))\simeq H^1(M;R\oim\pi \Omega_{P^*M}^{1,cl}(0))
\simeq H^1(M;\Omega_M^{1,cl})
\end{equation}

\begin{proposition}\label{pr:E-algebras}
If $\dim M \geq 2$, then the $\E$-algebras on $P^*M$ are classified 
by $H^1(P^*M;\E_{P^*M,1})\times H^1(M;\Omega_M^{1,cl})$.
\end{proposition}

\begin{proof}
From the diagram~\eqref{eq:pi}, one gets a commutative diagram

\begin{equation}\label{eq:t}
\xymatrix@C3em{
H^1(P^*M;\shaut[\C-alg](\E_{P^*M})) \ar[r] & H^1(P^*M;\Omega_{P^*M}^{1,cl}(0))\\
& H^1(M;\Omega_M^{1,cl}) \ar[lu]  \ar[u]. }
\end{equation}
Let $\dim M \geq 2$. The vertical arrow is an isomorphism by~\eqref{eq:R^1} and, since the map $s$ 
in~\eqref{eq:pi} is an isomorphism, the first map in the sequence~\eqref{eq:longexact} is injective. It follows that~\eqref{eq:longexact} splits.
\end{proof}
Recall that $H^1(M;\Omega_M^{1,cl})$ classifies $\C$-twisted line bundles\footnote{\label{note:tw}
A $\C^\times$-twisted line bundle on $M$ is given by a pair $(\stks,\shn)$, where $\stks$ is a $\C$-stack locally $\C$-equivalent to $\astk \C_M$ and $\shn$ is a global object of $\astk{\O_M}\tens[\C]\stks$ (for the notation $\astk{(\cdot)}$ refer to Section~\ref{section:class}). An isomorphism between $(\stks,\shn)$ and $(\stkt,\shm)$ is given by a $\C$-equivalence $\Psi\colon \stks\approxto\stkt$ of stacks and an isomorphism 
$\id_{\astk{\O_M}}\tens[\C]\Psi(\shn) \isoto \shm$ as global objects of $\astk{\O_M}\tens[\C]\stkt$. One checks that $\C^\times$-twisted line bundles are classified by $H^2(M;\C^\times_M \to \O^\times_M)\isoto[d\log]H^1(M;\Omega_M^{1,cl})$. Refer to~\cite{Kashiwara1989,D'Agnolo-Polesello2004} for more details.} on $M$. The diagonal arrow in~\eqref{eq:t} is thus given by
$$
[\shn] \mapsto [\E_\shn]
$$
(cf. Example~\ref{ex:E-alg} (i)). 
If $\dim M \geq 2$ then, up to isomorphism, any $\E$-algebra $\sha$ on $P^*M$ is of the form
$$
\opb {\pi}\shn\tens_{\opb{\pi}\O_M}\shendo[\E_{P^*M}^\op-mod](\shp) \tens_{\opb{\pi}\O_M} 
\opb{\pi}\shn^{\tens -1}
$$ 
for $\shp$ a locally free right $\E_{P^*M}$-module of rank $1$ which is filtered with trivial graded and $\shn$ a $\C$-twisted line bundle on $M$. In particular, $\sha$ is an inner form of 
$\E_{P^*M}$, \emph{i.e.} the glueing isomorphisms are inner automorphisms.

\begin{remark}\label{rk:Atiyah}
(i) From the isomorphism~\eqref{eq:TDO} it follows that $H^1(M;\Omega_M^{1,cl})$ classifies rings of twisted differential operators, tdo-rings for short, on $M$, \emph{i.e.} $\C$-algebras locally isomorphic to $\D_M$ (see for example~\cite{Kashiwara1989}). Note that, if $Y$ is a contact manifold with contact line bundle $\shl$, the Atiyah class $a(\shl)$ coincides with the class of the tdo-ring $\D_{\shl}=\shl\tens[\O_Y]\D_Y\tens[\O_Y]\shl^{\tens -1}\simeq \shendo[\D_Y^\op-mod](\shl\tens[\O_Y]\D_Y)$. 

(ii)  All the results in this section hold in the formal case, \emph{i.e.} by replacing the
$\E$-algebras with $\widehat\E$-algebras, which are obtained by using $\widehat\E_{P^*M}$ in Definition~\ref{definition:E-alg}.
\end{remark}

\section{Microdifferential algebroids on contact manifolds}\label{section:class}

In this section we will use the notion of linear stack (the classical reference is~\cite{Giraud1971}, a brief introduction may be found in~\cite{D'Agnolo-Polesello2004,D'Agnolo-Polesello2011}), that of cohomology with values in a crossed module or in a stack of 2-groups (see for example~\cite{Breen1994}, and~\cite{Polesello2008,D'Agnolo-Polesello2011} for a review), as well as that of filtered and graded stack, for which we refer to~\cite[Section 2]{Polesello2008}.

All the definitions and results in this section hold in the formal case.

We denote by $\astk{(\cdot)}$ the functor from the stack of algebras to that of linear stacks defined as follows (see~\cite[Section 2]{D'Agnolo-Polesello2011} for more details): 
if $\sha$ is an algebra, $\astk\sha$ is the linear stack associated to the separated pre-stack 
$$
U\mapsto \{ \text{category with a single object} \bullet  \text{and } \Endo(\bullet)=\sha(U)\};
$$
if $f\colon\sha\to\shb$ is an algebra morphism, $\astk f\colon \astk\sha\to\astk\shb$ is the
linear functor induced by that naturally defined at the level of pre-stacks.
Recall that the linear Yoneda embedding identifies $\astk\sha$ with the full sub-stack of right 
$\sha$-modules whose objects are locally free of rank $1$, and $\astk f$ with the functor induced by the extension of scalars $(\cdot)\tens[\sha]\shb$.
Note that for any linear functor $\Phi\colon\astk\sha\to\astk\shb$ there exist an open cover $\shu=\{U_{i}\}_{i\in I}$, morphisms of algebras $f_i\colon \sha|_{U_i}\to\shb|_{U_i}$ and invertible transformations 
$\Phi|_{U_i} \Rightarrow \astk f_i$.

\subsection{$(\E,\sigma)$-algebroids}\label{subsection:algebroids}
Let $Y$ be a complex  contact manifold.
It follows from Darboux's theorem and Proposition~\ref{pr:Q-alg} that any quantization 
algebra on $Y$ is an $\E$-algebra, \emph{i.e.} a $\C$-algebra locally isomorphic to 
$\E_V = \opb i \E_{P^*M}$ for any Darboux chart $i\colon Y\supset V \to P^*M$.
Though we cannot expect a globally defined $\E$-algebra on $Y$, we have:

\begin{theorem}[cf.~\cite{Kashiwara1996}]\label{th:existence}
There exists a canonical $\C$-stack $\stke_Y$ on $Y$ which is locally equivalent to 
$\astk{\E_V}$ for any Darboux chart $(V,i)$.
\end{theorem}
\noindent
Note that $\stke_Y$ is a $\C$-algebroid stack, \emph{i.e.} it is locally non-empty and
locally connected by isomorphisms (see~\cite{Kontsevich2001,D'Agnolo-Polesello2005} and~\cite{D'Agnolo-Polesello2011} for a more systematic study). 
Moreover, for any (locally defined) object $L$ of $\stke_Y$, the sheaf of endomorphisms 
$\shendo[\stke_Y](L)$ is a (locally defined) $\E$-algebra.

For a Darboux chart $i\colon Y\supset V \to P^*M$, set
$$
\sigma_V\colon\Gr(\E_V)\isoto\O^{hom}_Y|_V
$$
the pull-back on $V$ of the graded isomorphism~\eqref{eq:sigma}. 
The following proposition allows us to say that $\stke_Y$ provides a quantization of $Y$. 
\begin{proposition}
The $\C$-stack $\stke_Y$ is filtered and there is a canonical equivalence of graded stacks 
$$
{\pmb\sigma}\colon\stkGr(\stke_Y)\approxto \left(\O^{hom}_Y\right)^+
$$
which is locally isomorphic to $\astk{\sigma_V}$ for any Darboux chart $(V,i)$.
\end{proposition}

Then, it makes sense the following
\begin{definition}\label{definition:E-algebroids}
A $(\E,\sigma)$-algebroid on $Y$ is a filtered $\C$-stack $\stka$ locally equivalent to $\stke_Y$ as a 
$\C$-stack and endowed with an equivalence of graded stacks 
$$
{\pmb\nu}\colon\stkGr(\stka)\approxto \left(\O^{hom}_Y\right)^+
$$
A functor of $(\E,\sigma)$-algebroids $(\stka_1,\pmb\nu_1)\to(\stka_2,\pmb\nu_2)$ is a pair 
$(\Phi,\beta)$, where $\Phi\colon \stka_1 \to \stka_2$ is a filtered functor of $\C$-stacks 
and $\beta\colon \pmb \nu_2 \circ \stkGr (\Phi) \Rightarrow \pmb \nu_1$ is a 
graded invertible transformations of functors.
\end{definition}
Clearly, if $\sha$ is an $\E$-algebra, then $\astk{\sha}$ is a $(\E,\sigma)$-algebroid. 

\begin{remark}
By definition, a $(\E,\sigma)$-algebroid $\stka$ is locally equivalent to $\astk{\E_V}$ for any Darboux chart $(V,i)$. However, contrarily to the case of $\E$-algebras, 
this last characterization alone does not guarantee the triviality of the graded stack $\stkGr(\stka)$. Refer to~\cite{D'Agnolo-Polesello2011} for the classification of those $\C$-stacks which are locally equivalent to $\stke_Y$ without any assumption on the graded stack (called $\E$-algebroids in~\cite{D'Agnolo-Polesello2011}).
\end{remark}

Let us briefly recall how to describe the $(\E,\sigma)$-algebroids by means of non-abelian cocycles.
(Refer for example to~\cite[Appendix A]{D'Agnolo-Polesello2011} for a summary on non-abelian cocycles for algebroids and to~\cite{Breen1994} for the case of gerbes. See also~\cite{Toen2002} for a more conceptual approach.)

Let $\stka$ be a $(\E,\sigma)$-algebroid. By definition, there exists an open cover $Y=\Union\nolimits_{i\in I} V_i$ by Darboux charts $Y\supset V_i \to P^*M_i$ such that $\stka|_{V_i} \approx \astk{\E_{V_i}}$. Let $\Phi_i \colon \stka|_{V_i} \to \astk{\E_{V_i}}$ and 
$\Psi_i \colon \astk{\E_{V_i}}\to \stka|_{V_i}$ be $\C$-equivalences quasi-inverse one to each other.
On $V_{ij}=V_i\cap V_j$ there are equivalences 
$\Phi_{ij} = \Phi_i\circ\Psi_j \colon \astk{\E_{V_{ij}}} \to \astk{\E_{V_{ij}}}$, and on $V_{ijk}=V_{ij}\cap V_k$ 
there are invertible transformations 
$\alpha_{ijk} \colon \Phi_{ij}\circ\Phi_{jk}
\Rightarrow \Phi_{ik}$ such that on $V_{ijkl}=V_{ijk}\cap V_l $ the following diagram commutes
\begin{equation}
\label{eq:nat1}
\xymatrix@C5em{
\Phi_{ij}\circ\Phi_{jk}\circ\Phi_{kl} \ar@{=>}[r]^{\alpha_{ijk}
\bullet\id_{\Phi_{kl}}} \ar@{=>}[d]^{\id_{\Phi_{ij}}\bullet\alpha_{jkl}}
& \Phi_{ik}\circ\Phi_{kl} \ar@{=>}[d]^{\alpha_{ikl}} \\
\Phi_{ij}\circ\Phi_{jl} \ar@{=>}[r]^{\alpha_{ijl}} & \Phi_{il}
}
\end{equation}
(Here we denote by $\bullet$ the horizontal composition of transformations.)
Up to shrinking the open cover, one may suppose that $\Phi_{ij}$ is isomorphic to 
$\astk{\varphi_{ij}}$, for a $\C$-algebra isomorphism 
$\varphi_{ij}\colon \E _{V_j} \to \E_{V_i}$ on $V_{ij}$. 
On $V_{ijk}$ the invertible transformation $\astk{\varphi_{ij}}\circ\astk{\varphi_{jk}} \Rightarrow 
\astk{\varphi_{ik}}$ is thus given by an invertible operator $P_{ijk} \in \E_{V_{ijk}}$ of order 
$m_{ijk}$ satisfying
$$
\varphi_{ij}\circ\varphi_{jk}= \ad(P_{ijk}) \circ \varphi_{ik}
$$
Finally, on $V_{ijkl}$ the diagram \eqref{eq:nat1} corresponds to the 
equality
$$
P_{ijk} P_{ikl} = \varphi_{ij}(P_{jkl}) P_{ijl}
$$
Recall from~\cite[Proposition 2.5]{Polesello2008} that the algebroid $\stkGr(\stka)$ is described by the 2-cocycle 
$\{\sigma_{m_{ijk}}(P_{ijk})\}$ with values in $\O^{hom }_Y$, where $\sigma_{m_{ijk}}\colon F_{m_{ijk}}\E_{V_i}\to \shl^{\tens m_{ijk}}|_{V_i}$  denotes the symbol map of order $m_{ijk}$. 
Since $\stkGr(\stka)\approx \O^{hom \, +}_Y$, up to modify the $\varphi_{ij}$'s and the $P_{ijk}$'s by a coboundary, we may suppose that the $P_{ijk}$'s are of order $0$ with $\sigma_0(P_{ijk})=1$.
The non-abelian cocycle 
$$
(\{\E_{V_i}\},\{\varphi_{ij}\}, \{P_{ijk}\})
$$ 
is enough to reconstruct the $(\E,\sigma)$-algebroid $\stka$, up to equivalence.

\subsection{Classification} \label{subsection:class}

We denote by
$$
\{(\E,\sigma)\text{-algebroids}\}_Y
$$ 
the set of equivalence classes of $(\E,\sigma)$-algebroids on $Y$, pointed by the  class of $\stke_Y$.

Let $\stkAut[(\E,\sigma)](\stke_Y)$ denote the stack of 2-groups of $(\E,\sigma)$-algebroid auto-equivalences of $\stke_Y$.
More precisely, its objects are pair $(\Phi,\beta)$, where 
$\Phi\colon \stke_Y\to \stke_Y$ is a filtered equivalence of $\C$-stacks and 
$\beta\colon \pmb \sigma \circ \stkGr (\Phi) \Rightarrow \pmb \sigma$ is a 
graded invertible transformations of functors.
A morphism $\alpha\colon (\Phi_1,\beta_1)\Rightarrow (\Phi_2,\beta_2)$ is a 
filtered invertible transformation of functors 
$\alpha\colon \Phi_1 \Rightarrow \Phi_2$ making the following diagram 
commutative
\begin{equation}
\label{eq:nat}
\xymatrix@C3em{
\pmb \sigma \circ \stkGr(\Phi_1)\ar@{=>}[r]^-{\beta_1} 
\ar@{=>}[d]_-{\id_{\pmb\sigma}\circ \stkGr (\alpha)} & \pmb \sigma\\
\pmb \sigma\circ \stkGr({\Phi_2}) \ar@{=>}[ur]_-{\beta_2} &}
\end{equation}

Recall that one may define the first cohomology pointed set with values in a stack of 2-groups (cf.~\cite{Breen1994,Polesello2008,D'Agnolo-Polesello2011}). By definition of $(\E,\sigma)$-algebroid, one easily gets an isomorphism of pointed sets
\begin{equation}\label{eq:class2}
\{(\E,\sigma)\text{-algebroids}\}_Y\simeq H^1(Y;\stkAut[(\E,\sigma)](\stke_Y))
\end{equation}

\begin{theorem}\label{th:class}
There is an isomorphism of pointed sets
$$
\{(\E,\sigma)\text{-algebroids}\}_Y\simeq H^1(Y;\Omega_{Y}^{1,cl}(0))
$$
\end{theorem}

\begin{proof}
We will show that there is an equivalence of stacks of 2-groups
$$
\stkAut[(\E,\sigma)](\stke_Y)\approx \Omega_{Y}^{1,cl}(0)
$$
where the sheaf $\Omega_{Y}^{1,cl}(0)$ is considered as a discrete stack. 
Since equivalent stacks of 2-groups have isomorphic cohomologies, the result will follow 
from~\eqref{eq:class2}.

{\em First step.} Let us define a functor of stacks of 2-groups
$$
\pmb s\colon \stkAut[(\E,\sigma)](\stke_Y)\to \Omega_{Y}^{1,cl}(0)
$$ 
Take $(\Phi,\beta)\in \stkAut[(\E,\sigma)](\stke_Y)$ and choose a (locally defined) object $L$ of $\stke_Y$.
The isomorphism $\beta_L\colon  \pmb \sigma (\Phi(L)) \isoto \pmb \sigma(L)$ allows us to define
$$
F_m\shendo[\stke_Y](L)\ni p\mapsto   \pmb \sigma_{m-1}(\Phi(p) - p) \in \Gr_{m-1}\shendo[(\O^{hom}_Y)^+](\pmb \sigma (L)).
$$
This extends to a $\C$-derivation of $\shendo[(\O^{hom}_Y)^+](\pmb \sigma (L))\, \simeq \O^{hom}_Y$ of degree $-1$ preserving the Poisson structure.
By using the isomorphisms~\eqref{pr:GradDer}, we get a local section $\{\pmb s (\Phi,\beta)\}_{L}\in\Omega_Y^{1,cl}(0)$, which depends only on the isomorphism class of $L$. Since any two objects $L,L'$ of $\stke_Y$ are locally isomorphic, it follows that $\{\pmb s (\Phi,\beta)\}_{L}=\{\pmb s (\Phi,\beta)\}_{L'}$ locally, hence globally. This defines $\pmb s (\Phi,\beta)\in\Omega_Y^{1,cl}(0)$, which depends only on the  isomorphism class of $(\Phi,\beta)$. 

{\em Second step.} At any point $y\in Y$, we may find a Darboux chart $(V,i)$ containing $y$ such that $\stke_Y|_V\approx \E _V^+$. Let $(\Phi,\beta)\in  \stkAut[(\E,\sigma)](\E_V^+)$.
Up to take a smaller open subset, one may suppose that $\pmb\sigma$ is 
isomorphic to $\astk{\sigma_V}$ and $\Phi$ to $\astk{\varphi}$
for a $\C$-algebra automorphism $\varphi\colon \E _V \to \E_V$.
The graded transformation  $\beta$ is thus given by a nowhere vanishing function 
$b\in\O_Y|_V\simeq\Gr_0(\E_V)$. 

Locally there exists an invertible operator $Q\in F_0\E_V$ such that $b\sigma_{V,0} (Q)=1$. 
Set $\tilde\varphi= \ad(\opb Q) \circ\varphi$. Then $Q$ defines a morphism 
$(\tilde\varphi^+,\id) \Rightarrow (\Phi, \beta)$.
It follows that the functor of stacks of 2-groups
$$
\left[\E_{V,1} \to[\ad]  \shaut[\C-alg](\E_V) \right]
\longrightarrow \stkAut[(\E,\sigma)](\E_V^+) \qquad
\psi \mapsto (\psi^+,\id)
$$
is essentially surjective, where the left-hand side denotes the stack of 2-groups associated 
to the crossed module $\E_{V, 1}=\{P\in F_0\E_V; \mbox{ $\sigma_{V,0}(P)=1$}\} \to[\ad]  \shaut[\C-alg](\E_V)$
(see~\cite{Breen1994,Polesello2008,D'Agnolo-Polesello2011}).

Let $\psi_1,\psi_2$ be two sections of  $\shaut[\C-alg](\E_V)$ and consider a morphism
$\alpha\colon(\psi_1^+,\id)\Rightarrow (\psi_2^+,\id)$ in $\stkAut[(\E,\sigma)](\E_V^+)$. 
Then $\alpha$ is locally given by an invertible operator $R\in F_0\E_V$ satisfying
$$
\psi_2 = \ad (R) \circ \psi_1
$$ 
where the diagram~\eqref{eq:nat} corresponds to the equality $\sigma_{V,0}(R)=1$. 
Therefore, $R$ defines a morphism $\psi_1\to\psi_2$ in 
$\left[\E_{V,1} \to[\ad]  \shaut[\C-alg](\E_V) \right]$. 
The above functor is thus full. Since it is also faithful, we get an 
equivalence
$$
\left[\E_{V,1} \to[\ad]  \shaut[\C-alg](\E_V)\right]\approx \stkAut[(\E,\sigma)](\E_V^+)
$$

{\em Third step.} 
Through the above equivalence, the functor
$$
\pmb s|_V \colon \stkAut[(\E,\sigma)](\stke_Y|_V) \to \gamma_*\Omega_{\opb \gamma(V)}^{1,cl}(0)
$$
reduces to the functor
$$
\left[\E_{V,1} \to[\ad]  \shaut[\C-alg](\E_V)\right]\to
\gamma_*\Omega_{\opb \gamma(V)}^{1,cl}(0)
$$
induced by the morphism~\eqref{eq:key}.
By Lemma~\ref{lemma:key} and~\cite[Proposition 4.4]{Polesello2008} it follows that this is an equivalence of stacks of 2-groups. The functor $\pmb s$ is thus locally, hence globally, an equivalence.
\end{proof}

\begin{corollary}\label{cor:exact}
Suppose that $Y$ is exact. Then
$$
\{(\E,\sigma)\text{-algebroids}\}_Y\simeq H^1(Y;\Omega_{Y}^{1,cl}) \times  H^1(Y;\C_Y)
$$
If moreover $H^i(Y;\O_Y)=0$ for $i=1,2$, then 
$$
\{(\E,\sigma)\text{-algebroids}\}_Y\simeq H^2(Y;\C_{Y})\times H^1(Y;\C_Y)
$$
\end{corollary}

\begin{proof}
If $Y$ is exact, then the associated contact line bundle $\shl$ is trivial and the Atiyah class $a(\shl)$ vanishes. Hence, by Lemma~\ref{lemma:Atiyah} the sequence~\eqref{eq:OmegaY} splits and one uses Theorem~\ref{th:class}.
The second part of the statement follows by taking the long exact sequence associated to~\eqref{eq:deRham}.
\end{proof}

From the isomorphism~\eqref{eq:R^1} we immediately get
\begin{corollary}\label{cor:Boutet}
Let $Y=P^*M$ with $\dim M \geq 2$. Then 
$$
\{(\E,\sigma)\text{-algebroids}\}_{P^*M}\simeq H^1(M;\Omega_M^{1,cl})
$$ 
\end{corollary}

It remains to consider the 1-dimensional case (cf.~\cite[Section 3.5]{BoutetdeMonvel1999}).
Let $Y$ be a Riemann surface. Recall from Remark~\ref{rk:1-dim} that $Y\simeq P^*Y$ has a canonical contact structure, whose contact line bundle is $\Theta_Y$.

If $Y$ is non-compact, then 
the contact structure is exact. Since $H^1(Y;\Omega_Y) \simeq H^1(Y;\O_Y)=0$, it follows from Corollary~\ref{cor:exact} that the 
$(\E,\sigma)$-algebroids are classified by $H^1(Y;\C_Y)$.

Suppose now that $Y$ is connected and compact of genus $g$. Then $H^1(Y;\Omega_Y)\simeq\C$ by the residue map.

If $g=1$, then $Y$ is exact.
Again, from Corollary~\ref{cor:exact} it follows that the $(\E,\sigma)$-algebroids are classified by 
$$
H^1(Y;\Omega_{Y})\times H^1(Y;\C_Y)\simeq \C^3
$$

If  $g\neq 1$, then the $(\E,\sigma)$-algebroids are classified by
$$
H^1(Y;\Omega^{1,cl}_Y(0))\isoto H^1(Y;\C_Y)\simeq \C^{2g}
$$
where the first isomorphism is obtained by using the sequence~\eqref{eq:longOmega} below. In particular, up to equivalence, $\astk\E_{\mathbb P^{1}}$ is the unique $(\E,\sigma)$-algebroid on $\mathbb P^{1}$ 
(cf. Example~\ref{ex:q-proj}).

\begin{remark}
Using similar techniques, one may show that Theorems~\ref{th:existence} and~\ref{th:class} hold also in the framework of $C^\infty$ manifolds, hence replacing $\E$-algebras with Toeplitz algebras (see~\cite{BoutetdeMonvel2002}). However, in that case any quantization algebroid comes from a unique, up to isomorphism, Toeplitz algebra (cf.~\cite{Deligne1995}). It follows that in the  $C^\infty$ case there is a one-to-one correspondence between quantization algebras and quantization algebroids. Note also that in the $C^\infty$ case the classification of these objects is always given by $H^2(Y;\C_Y)\times  H^1(Y;\C_Y)$, since the exact sequence~\eqref{eq:OmegaY} splits and $H^1(Y;\Omega_Y^{1,cl}) \isoto H^2(Y;\C_Y)$.
\end{remark}

\subsection{$(\E,\sigma)$-algebroids on compact K\"ahler manifolds}
Let $Y$ be a complex  contact manifold of dimension $2n+1$, with associated contact line bundle 
$\shl$.
Consider the long exact sequence in cohomology associated to~\eqref{eq:OmegaY}
\begin{equation}\label{eq:longOmega}
H^0(Y;\C_Y) \to[\delta] H^1(Y;\Omega_Y^{1,cl}) \to H^1(Y;\Omega_{Y}^{1,cl}(0)) 
\to H^1(Y;\C_Y)
\end{equation}
The coboundary map $\delta$ sends $\lambda\in H^0(Y;\C_Y)$ to $-\lambda a(\shl)=[\D_{\shl^{\tens -\lambda}}]$, where $a(\shl)$ is the Atiyah class of $\shl$ (see Remark~\ref{rk:Atiyah} (i)).
In particular, from the isomorphism $\Omega_Y\simeq \shl^{\tens -(n+1)}$ it follows that the class of the tdo-ring $\D_{\Omega^{\tens 1/2}_Y}$  
is sent to the class of the $(\E,\sigma)$-algebroid $\stke_Y$ in $H^1(Y;\Omega_Y^{1,cl}(0))$.
Note that, up to isomorphism/equivalence, they are the unique ones which carry an anti-involution corresponding to the operation of taking the formal adjoint of (micro)differential operators (see~\cite{Polesello2007}).

Recall that, if $Y$ is compact, then $H^i(Y;\C_Y)$ is finite dimensional  for any $i$ and one sets 
$b_i(Y)=\dim H^i(Y;\C_Y)$. Given a finite rank vector bundle $\shv$ on $Y$, let $c_1(\shv) \in H^2(Y;\C_Y)$ be its first Chern class. Then, $c_1(\shl)= \frac{1}{n+1}c_1(\Theta_Y)$.

\begin{lemma}[cf.~{\cite[Lemma 1]{Ye1994}}]\label{lem:Ye}
If $Y$ is compact K\"ahler\footnote{By~\cite[Theorem 2.5]{Frantzen-Peternell2010}, this result still holds by replacing "K\"ahler" with "bimeromorphically equivalent to a K\"ahler manifold''.} of $\dim Y\geq 3$, then $c_1(\shl)\neq0$.
\end{lemma}

In particular, in the above situation $Y$ cannot be exact.

\begin{proposition}\label{pr:Kaheler}
Let $Y$ be compact K\"ahler of $\dim Y\geq 3$ with $b_1(Y)=0$. Then
$$
\{(\E,\sigma)\text{-algebroids}\}_Y \simeq H^1(Y;\Omega_Y^{1,cl})/ H^0(Y;\C_Y)
$$
\end{proposition}

\begin{proof}
Recall that $\beta(\frac{i}{2\pi}a(\shv))=c_1(\shv)$ for any line bundle $\shv$, where
$\beta\colon H^1(Y;\Omega_Y^{1,cl})\to H^2(Y;\C_Y)$ is the coboundary map induced by the exterior differential. By Lemma~\ref{lem:Ye}, one gets $a(\shl)\neq 0$ in $H^1(Y;\Omega_Y^{1,cl})$, so that the map $\delta$ in~\eqref{eq:longOmega} is non-trivial, hence injective. One then uses Theorem~\ref{th:class}.
\end{proof}
Note that, for $Y$ compact K\"ahler of $\dim Y\geq 3$, one also has
$$
\sect(Y;\stkAut[(\E,\sigma)](\stke_Y))\approxto[\pmb s] H^0(Y;\Omega_{Y}^{1,cl}(0)) \isofrom H^0(Y;\Omega_Y^{1,cl}) \simeq H^{1,0}(Y)
$$ 
where the last isomorphism follows from the fact that every global holomorphic form on $Y$ is closed.

For contact Fano manifolds the previous classification becomes very simple.
Recall that a Fano manifold is a compact complex manifold such that the dual of the canonical bundle is ample.

\begin{proposition}\label{pr:Fano}
Let $Y$ be connected and Fano of $\dim Y \geq 3$.
\begin{itemize}
\item[(i)] If $b_2(Y) \geq 2$, then
$$
\{(\E,\sigma)\text{-algebroids}\}_Y \simeq \C
$$
\item[(ii)] If $b_2(Y) = 1$, then, up to equivalence, $\stke_Y$ is the unique $(\E,\sigma)$-algebroid on $Y$.
\end{itemize}
\end{proposition}

\begin{proof}
Since $Y$ is compact and connected, then $H^0(Y;\O_Y)\simeq\C$. By classical results on Fano manifolds $H^i(Y;\O_Y)=0$ for $i\geq 1$ hence, using the  long exact sequence in cohomology obtained from~\eqref{eq:deRham},
one gets $H^i(Y;\Omega_Y^{1,cl})\simeq H^{i+1}(Y;\C_Y)$ for $i\geq 0$. 
Moreover $b_1(Y)=0$, as $Y$ is simply connected.

Being projective, Fano manifolds are K\"ahler. By Proposition~\ref{pr:Kaheler}, one thus gets an isomorphism 
$$
\{(\E,\sigma)\text{-algebroids}\}_Y \simeq \C^{b_2(Y)-1}
$$ 
and (ii) follows.
If $b_2(Y)\geq 2$, then $Y$ is contact isomorphic to $P^*\mathbb P$ for some complex projective space $\mathbb P$ (see~\cite[Corollary 4.2]{LeBrun-Salamon1994}). Since $b_2(P^*\mathbb P)= b_2(\mathbb P) + 1$ by the Leray-Hirsch theorem and $b_2(\mathbb P)=1$, then (i) follows.
\end{proof}

\begin{example}\label{ex:q-proj}
The complex projective space $\mathbb P^{2n+1}$ is a contact Fano 
manifold (cf. Example~\ref{ex:contact} (i)) with $b_2(\mathbb P^{2n+1})=1$. By Proposition~\ref{pr:Fano}, up to equivalence, $\stke_{\mathbb P^{2n+1}}$ is the unique $(\E,\sigma)$-algebroid on $\mathbb P^{2n+1}$ for $n\geq 1$.
\end{example}

By a result of Demailly~\cite{Demailly2002}, any complex contact manifold $Y$ 
which is compact K\"ahler with $b_2(Y)=1$ is Fano. Hence:

\begin{corollary}
Let $Y$ be connected, compact K\"ahler of $\dim Y\geq 3$ with $b_2(Y)=1$. Then, up to equivalence, $\stke_Y$ is the unique $(\E,\sigma)$-algebroid on $Y$.
\end{corollary}

If instead $b_2(Y)\geq 2$ and $Y$ projective, we may use the following result to reduce the classification of $(\E,\sigma)$-algebroids to Corollary~\ref{cor:Boutet}.

\begin{theorem}[cf.~\cite{KPSW2000,Demailly2002}]\label{th:DKPSW}
If $Y$ is projective with $b_2(Y)\geq 2$, then it is contact isomorphic to $P^*M$ for some 
projective manifold $M$.
\end{theorem}
In fact, according to~\cite{Frantzen-Peternell2010}, the same result holds replacing "projective" with "compact K\"ahler threefold".

\section{$\E$-algebras vs. $(\E,\sigma)$-algebroids}\label{section:vs}

Let $Y$ be a complex  contact manifold of dimension $2n+1$, with associated contact line bundle 
$\shl$.

Recall that the functor $\astk{(\cdot)}$ from algebras to linear stacks defines a map
\begin{equation}\label{eq:astk}
(\cdot)^+\colon \{\E \text{-algebras}\}_Y\to \{(\E,\sigma)\text{-algebroids}\}_Y
\end{equation}
where the left-hand term is the (possibly empty) set of isomorphism classes of 
$\E$-algebras on $Y$.

Let $\stka$ be a $(\E,\sigma)$-algebroid. Given a global object $L\in \stka(Y)$, the sheaf 
$\shendo[\stka](L)$ of its endomorphisms is an $\E$-algebra, and the fully faithful functor 
$\astk{\shendo[\stka](L)}\to\stka$ is an equivalence of $(\E,\sigma)$-algebroids.
It follows that~~\eqref{eq:astk} is surjective if and only if any $(\E,\sigma)$-algebroid 
has a global object. In this situation, we may make the set $\{\E \text{-algebras}\}_Y$ pointed by the 
class of any $\E$-algebra $\sha$ satisfying $\astk \sha \approx\stke_Y$. Note that, by~\cite[Corollary 3.4]{Polesello2007}, we may always choose $\sha$ endowed with an anti-involution, locally corresponding to the operation of taking the formal adjoint of microdifferential operators.

Let $Y=P^*M$, for a complex manifold $M$. By definition of the functor $\pmb s$ in the proof of Theorem~\ref{th:class}, the map~\eqref{eq:astk} identifies with the morphism
$$
H^1(P^*M;\shaut[\C-alg](\E_{P^*M})) \to H^1(P^*M;\Omega_{P^*M}^{1,cl}(0))
$$
in~\eqref{eq:longexact}. If $\dim M \geq 2$, by the commutative diagram~\eqref{eq:t} and the isomorphism~\eqref{eq:R^1}, we get that, up to equivalence, any $(\E,\sigma)$-algebroid on $P^*M$ is of the form $\astk{\E_{\shn}}$ for a unique $\C^\times$-twisted line bundle $\shn$ on $M$. Hence:

\begin{proposition}\label{prop:Boutet}
Let $Y=P^*M$ with $\dim M \geq 2$. Then~\eqref{eq:astk} is surjective.
\end{proposition}
More precisely, it follows from Proposition~\ref{pr:E-algebras} that in the above situation the map~\eqref{eq:astk} 
identifies with the projection
$$
H^1(P^*M;\E_{P^*M,1})\times H^1(M;\Omega_M^{1,cl}) \to H^1(M;\Omega_M^{1,cl})
$$

Combining the above result with Theorem~\ref{th:DKPSW}, we get
\begin{corollary}\label{cor:DKPSW}
Let $Y$ be projective of $\dim Y\geq 3$ (or a compact K\"ahler threefold) with $b_2(Y)\geq 2$. 
Then, for any $(\E,\sigma)$-algebroid $\stka$, there exists an $\E$-algebra 
$\sha$ such that $\stka\approx \sha^+$.
\end{corollary}

\subsection{The formal case}
Even when $\E$-algebras exist, their classification in general differs from that of
$(\E,\sigma)$-algebroids.
In the formal case ({\em i.e.} dropping the growth condition~\eqref{eq:estmicrod}),
the following criterion\footnote{Conjecturally, this criterion works also in the 
holomorphic case.} gives conditions for the existence of 
$\widehat\E$-algebras (\emph{i.e.} formal $\E$-algebras) and to ensure
that their classification coincides with that of $(\widehat\E,\sigma)$-algebroids 
(\emph{i.e.} formal $(\E,\sigma)$-algebroids). 

\begin{theorem}\label{th:stkVSsh}
\begin{itemize}

\item[(i)] If $H^2(Y;\shl^{\tens k})=0$ for any $k\leq -1$, then for any 
$(\widehat\E,\sigma)$-algebroid $\stka$ there exists an $\widehat\E$-algebra $\sha$ such that 
$\stka\approx \sha^+$.

\item[(ii)] If moreover  $H^1(Y;\shl^{\tens k})=0$ for any $k\leq -1$, there is an isomorphism of pointed sets
$$
(\cdot)^+\colon\{ \widehat\E \text{-algebras}\}_Y\isoto  \{(\widehat\E,\sigma)\text{-algebroids}\}_Y
$$
\end{itemize}
\end{theorem}

If $Y$ is exact, then the contact line bundle $\shl$ is trivial, so that the conditions (i) and (ii) in 
Theorem~\ref{th:stkVSsh} read as $H^2(Y;\O_Y)=0$ and $H^1(Y;\O_Y)=0$, respectively, and 
one may thus use (the formal analogue of) Corollary~\ref{cor:exact}. We refer 
to~\cite{Nest-Tsygan2001,Kontsevich2001,Yekutieli2005,Kaledin2006,Calaque-Halbout2011} for similar hypothesis in the framework of complex/algebraic symplectic/Poisson manifolds.

\begin{proof}[Proof of Theorem~\ref{th:stkVSsh}]
For any Darboux chart $i\colon Y\supset V \to P^*M$, set $\widehat\E_V=\opb i \widehat\E_{P^*M}$ and let $\sigma_{m} \colon F_{m}\widehat\E_{V}\to \shl^{\tens m}|_{V}$  denote the symbol map of order $m$. Set $\widehat\E_{V,1} = \{P\in F_0\widehat\E_V; \mbox{ $\sigma_0(P)=1$}\}\subset F_0\widehat\E_V^\times$. Then one easily checks that there is a well defined surjective morphism of groups
\begin{equation}
\label{eq:sub}
\tilde\sigma_{-1}\colon \widehat\E_{V,1} \to \shl^{\tens -1}, \quad P\mapsto \sigma_{-1}(P-1)
\end{equation}
If $\tilde\sigma_{-1}(P)=0$, then $P-1\in F_{-2}\widehat\E_V$ and one may define $\tilde\sigma_{-2}(P)=
\sigma_{-2}(P-1)\in \shl^{\tens -2}$. Recursively, let $m\in \Z_{<0}$ and suppose that $\tilde\sigma_{i}(P)=0$ for $-1\leq i \leq m+1$. Then $P-1\in F_m\widehat\E_V$ and one may define $\tilde\sigma_m(P)=\sigma_m(P-1)\in \shl^{\tens m}$. Clearly, $P=1$ if and only if $\tilde\sigma_m(P)=0$ for all $m\in \Z_{<0}$.

Let $\stka$ be a $(\widehat\E,\sigma)$-algebroid on $Y$. Recall from 
Section~\ref{subsection:algebroids} that it is described by a non-abelian cocycle $(\{\widehat\E_{V_i}\},\{\varphi_{ij}\}, \{P_{ijk}\})$, where $\shu = \{V_i\}_{i\in I}$ is an open cover of $Y$ by Darboux charts, $\varphi_{ij}\colon \widehat\E_{V_j}\to\widehat\E_{V_i}$ are isomorphisms of $\C$-algebras on $V_{ij}=V_i\cap V_j$ and $P_{ijk}\in\sect(V_{ijk};\widehat\E_{V_i,1})$ are such that
\begin{equation}
\label{eq:KasCoc}
\begin{cases}
\varphi_{ij}\circ \varphi_{jk} = \ad(P_{ijk})\circ \varphi_{ik}\\ 
P_{ijk} P_{ikl} = \varphi_{ij}(P_{jkl}) P_{ijl}
\end{cases}
\end{equation}
By applying $\tilde\sigma_{-1}$ to the second equation, we get a 2-cocycle $\{\tilde\sigma_{-1}(P_{ijk})\}$ with values in $\shl^{\tens -1}$. By hypothesis, $H^2(Y;\shl^{\tens -1})=0$. Therefore, up to shrink the cover, there exist operators $Q_{ij}\in\sect(V_{ij};\widehat\E_{V_i,1})$ satisfying $\tilde\sigma_{-1}(P_{ijk}) = \tilde\sigma_{-1}(Q_{jk}) + \tilde\sigma_{-1}(Q_{ij}) - \tilde\sigma_{-1}(Q_{ik})$ on $V_{ijk}$. 
Set 
$$
\begin{cases}
\varphi^1_{ij}=\ad ( \opb Q_{ij})\circ \varphi_{ij}\\ 
P^1_{ijk} = \varphi^1_{ij}(\opb Q_{jk}) \opb Q_{ij}P_{ijk}Q_{ik}
\end{cases}
$$
This gives a non-abelian cocycle $(\{\widehat\E_{V_i}\},\{\varphi^1_{ij}\}, \{P^1_{ijk}\})$ equivalent  
to~\eqref{eq:KasCoc} and satisfying  $\sigma_0(P^1_{ijk})=\tilde\sigma_{-1}(P^1_{ijk})=0$. (Refer for example to~\cite[Appendix A]{D'Agnolo-Polesello2011} for a summary on non-abelian cocycles for algebroids). We may thus define the 2-cocycle $\{\tilde\sigma_{-2}(P^1_{ijk})\}$ with values in $\shl^{\tens -2}$. By induction,  one gets that the non-abelian cocycle~\eqref{eq:KasCoc} is finally equivalent to a non-abelian cocycle of the form $(\{\widehat\E_{V_i}\},\{\tilde\varphi_{ij}\}, \tilde P_{ijk})$ with $\tilde\sigma_{m}(\tilde P_{ijk})=0$ for any $m\in  \Z_{<0}$. It follows that $\tilde P_{ijk}=1$, hence this represents an $\widehat\E$-algebra $\sha$ satisfying $\astk \sha\approx \stka$.

(ii) Let $\sha$ (resp. $\shb$) be an $\widehat\E$-algebra, and $\nu$ (resp. $\upsilon$) the corresponding graded isomorphism~\eqref{eq:symbol}. Let $\Phi\colon \astk{\sha} \to \astk{\shb}$ be an equivalence of $(\widehat\E,\sigma)$-algebroids. It is enough to show that there is a $\C$-algebra isomorphism\footnote{In fact, the proof will produce a $\C$-algebra isomorphism 
$f\colon \sha \isoto \shb$ and a filtered invertible transformation of functors 
$\alpha\colon \Phi \Rightarrow \astk f$ making the following diagram 
commutative
$$
\xymatrix@C3em{
\astk \upsilon \circ \stkGr(\Phi)\ar@{=>}[r]^-{\sim} 
\ar@{=>}[d]_-{\id_{\astk \upsilon}\circ \stkGr (\alpha)} & \astk\nu\\
\astk \upsilon\circ \stkGr(\astk f) \ar@{=>}[ur]_-{\sim} &}
$$} 
$\sha \isoto \shb$

Recall that $\astk\sha \subset \stkMod(\sha^\op)$ is identified to the sub-stack of locally  free modules of rank 1. Set $\shp =\Phi(\sha)$. Then $\shp$ is a locally  free $\shb^\op$-module of rank $1$ and $\Phi$ induces an isomorphism $\sha\isoto\shendo[\shb^\op](\shp)$ of $\C$-algebras. Moreover, since by definition $\Phi$ is filtered and there is a graded invertible transformation of functors 
$\astk \upsilon \circ \stkGr (\Phi) \Rightarrow \astk\nu$, it follows that $\shp$ is locally free of rank $1$ as filtered module, and $\Gr(\shp)$ is free as $\O^{hom}_Y$-module.

Let $\upsilon_{m} \colon F_{m}\shb\to \shl^{\tens m}$  denote the symbol map of order $m$ and
set $\shb_1 = \{P\in F_0\shb; \mbox{ $\upsilon_0(P)=1$}\}\subset F_0\shb^\times$.
It follows that $\shp$ is described by a 1-cocycle $\{P_{ij}\}$ with values in  $\shb_1$ and one may define the 1-cocycle $\{\tilde\upsilon_{-1}(P_{ij})=\upsilon_{0}(P_{ij}-1)\}$ with values in $\shl^{\tens -1}$. Then one proceeds in a similar way as in (i) and gets that $\{P_{ij}\}$ is cohomologous to the trivial cocycle. $\shp$ is thus free of rank $1$ and $\shendo[\shb^\op](\shp)\simeq \shb$ as $\C$-algebras.
\end{proof}

\begin{remark}
For $\stka=\stke_Y$, the construction in the proof of Theorem~\ref{th:stkVSsh} (i) produces a class in $H^2(Y;\shl^{\tens -1})\simeq H^2(Y;  \Omega_{Y}^{1,cl}(-1))$ which gives an obstruction to the existence of an $\E$-algebra $\sha$ satisfying $\astk \sha\approx \stke_Y$. This has to be compared with the first Rozansky-Witten class of the canonical symplectification $\widetilde Y$ as defined in~\cite{Kapranov1999} (see~\cite{BGNT2007} for the symplectic case).
\end{remark}

\begin{proposition}
If $Y$ is Fano of $\dim Y \geq 3$, then   
$$
\{ \widehat\E \text{-algebras}\}_Y\simeq \{(\widehat\E,\sigma)\text{-algebroids}\}_Y
$$
\end{proposition}

\begin{proof}
Let $Y$ be compact. By definition, $Y$ is Fano if the dual of $\Omega_Y\simeq\shl^{\tens -(n+1)}$
is ample, hence if and only if $\shl$ is ample. By the Kodaira vanishing theorem, one gets 
$H^i(Y;\shl^{\tens k})=0$ for any $k\leq -1$ and $i \leq\dim Y -1$. The result then follows from 
Theorem~\ref{th:stkVSsh}.
\end{proof}

In particular, the above isomorphism holds if $Y$ is compact K\"ahler of $\dim Y\geq 3$ with 
$b_2(Y)=1$, since it is Fano (see~\cite[Corollary 3]{Demailly2002}). If moreover $Y$ is connected, from (the formal analogue of) Proposition~\ref{pr:Fano} (ii) it follows that these sets are singletons.

\begin{theorem}\label{th:Boutet2}
Let $M$ be a complex manifold. The morphism of pointed sets
$$
(\cdot)^+\colon \{ \widehat\E \text{-algebras}\}_{P^*M}\to
\{ \widehat\E \text{-algebroids}\}_{P^*M}
$$
is surjective. If moreover $\dim M \geq 3$, then it is an isomorphism.
\end{theorem}

\begin{proof}
If $\dim M \geq 2$, the map $(\cdot)^+$ is surjective by (the formal analogue of) 
Proposition~\ref{prop:Boutet}. If $\dim M =1$, then $P^*M\simeq M$ and $\O_{M}(1)$ identifies with the 
sheaf $\Theta_M$ of holomorphic vector fields. Since $H^2(M;\Theta_M^{\tens k})=0$ for any $k$, 
by Theorem~\ref{th:stkVSsh} (i) the map $(\cdot)^+$ is surjective.

Set $\dim M=n+1$ and let $\pi\colon P^*M\to M$ be the projection. Recall that one has
$$
 R\pi_* \O_{P^*M}(k)\simeq \left\{ \begin{array}{l}
     S^{k}_{\O_M}(\Theta_M)\mbox{ if $k\geq 0$}\\
    S^{-k-n-1}_{\O_M}(\Omega^1_M)\tens \Omega_M[-n]\mbox{ if $k\leq -1$}
     \end{array}\right.
$$
where $S^i_{\O_M}$ denotes the $i$-th symmetric product over $\O_M$ ($=0$ for $i\leq -1$).
If $k\leq -1$, we thus have 
$$
H^i(P^*M;\O_{P^*M}(k))\simeq H^i(M;R\pi_* \O_{P^*M}(k))\simeq 
H^{i-n}(M;S^{-k-n-1}_{\O_M}(\Omega^1_M)\tens \Omega_M)
$$
It follows that $H^i(P^*M;\O_{P^*M}(k))=0$ for any $k\leq -1$ and $i\leq n-1$. 
If $\dim M\geq 3$, then $(\cdot)^+$ is an isomorphism by Theorem~\ref{th:stkVSsh} (ii).
\end{proof}
Thanks to (the formal analogue of) Corollary~\ref{cor:Boutet}, one recovers that the $\widehat\E$-algebras on $P^*M$ are classified by $H^1(M;\Omega_M^{1,cl})$ if $\dim M \geq 3$ 
(cf.~\cite[Theorem 1]{BoutetdeMonvel1999}, where this is proved by using (the formal analogue of) Proposition~\ref{pr:E-algebras} and by showing that $H^1(P^*M;\widehat\E_{P^*M,1})=0$). In particular, any $\widehat\E$-algebra is of the form $\widehat\E_{\shn}$ for a $\C^\times$-twisted line bundle $\shn$ on $M$.

By using Theorem~\ref{th:DKPSW}, one gets
\begin{corollary}\label{cor:DKPSW2}
Let $Y$ be projective of $\dim Y\geq 5$ with $b_2(Y)\geq 2$. Then   
$$
\{ \widehat\E \text{-algebras}\}_Y\simeq \{(\widehat\E,\sigma)\text{-algebroids}\}_Y
$$
\end{corollary}

The ``exotic phenomena", \emph{i.e.} the existence of non-isomorphic $\widehat\E$-algebras giving rise to equivalent $(\widehat\E,\sigma)$-algebroids, appear if $\dim Y \leq 3$. Precisely, for $Y\simeq P^*M$ with $M$ a compact K\"ahler surface different from $\mathbb P^2$ (in which case $Y$ would be Fano), and for curves. Several computations for these cases may be found in~\cite{BoutetdeMonvel1999}. Let us briefly consider the 1-dimensional case\footnote{Note that the original manuscript contains an error. See~http://www.math.jussieu.fr/$\sim$boutet/ComplexStar.pdf for a corrected version.}.

Let $Y$ be a connected compact Riemann surface of genus $g$.
Recall from Remark~\ref{rk:1-dim} that the contact line bundle is $\Theta_Y$.
By Theorem~\ref{th:Boutet2},  for any $(\widehat\E,\sigma)$-algebroid $\stka$ there 
exists an $\widehat\E$-algebra $\sha$ such that $\stka\approx \sha^+$.
However, the classification of the $(\widehat\E,\sigma)$-algebroids (which is equivalent to that given at the end of Section~\ref{subsection:class} for the holomorphic case) differs from that of 
the $\widehat\E$-algebras, since $H^1(Y;\Theta_Y^{\tens k})$ is in general 
not zero for $k\leq -1$ (in particular $H^1(Y;\Theta_Y^{\tens -1}) \simeq\C$ by the residue map). 

Let $g\geq 2$. Since $H^1(Y;\Theta_Y^{\tens k})=0$ for any $k\leq -2$, the morphism~\eqref{eq:sub} induces $H^1(Y;\widehat\E_{Y,1})\simeq H^1(Y;\Theta_Y^{\tens -1})\simeq\C$ and the exact
sequence~\eqref{eq:longexact} reduces to
$$
0 \to \C \to \{ \widehat\E \text{-algebras}\}_Y \to \C^{2g} \to 0
$$
Since it splits, it follows that the $\widehat\E$-algebras are classified by $\C^{2g+1}$.

Let $g=0$. Then $Y\simeq \mathbb P^1$. Since $H^i(\mathbb P^1;\Omega_{\mathbb P^1}(0))=0$
for $i=0,1$, by the exact sequence~\eqref{eq:longexact} it follows that the $\widehat\E$-algebras are classified by $H^1(\mathbb P^1;\widehat\E_{\mathbb P^1, 1})$. This is far to be 
trivial, since $H^1(\mathbb P^1;\Theta_{\mathbb P^1}^{k})\simeq H^1(\mathbb P^1;\O_{\mathbb P^1}(2k))$ is never vanishing for $k\leq -1$.

\providecommand{\bysame}{\leavevmode\hbox to3em{\hrulefill}\thinspace}

\end{document}